%% file: FiniteCurves.tex
\documentclass[a4paper,11pt, english]{amsart}

\usepackage{amssymb}
\usepackage{amsfonts}
\usepackage{amsmath}

\usepackage{amsthm}
\usepackage[utf8]{inputenc}
\usepackage{pgf}
\usepackage{graphicx}
\usepackage{tikz}
\usepackage{pdfpages}
\usepackage[all]{xy}
\usepackage{verbatim,enumerate}
\usepackage{hyperref}
\usetikzlibrary{patterns}
\usepackage[yyyymmdd,hhmmss]{datetime}

\usepackage{color}

\def\1{\color{blue}}
\def\2{\color{red}}
\def\Dg:{{\bf Dg:\enspace}\ignorespaces}

\def\wt#1{\|#1\|}
\def\val{\operatorname{val}}
\def\Pic{\operatorname{Pic}}
\def\mult{\operatorname{mult}}
\def\oval{\mathfrak o}

\let\sminus\smallsetminus

\newtheorem{thm}{Theorem}[section]
\newtheorem{lemma}[thm]{Lemma}
\newtheorem{prop}[thm]{Proposition}
\newtheorem{cor}[thm]{Corollary}

{\theoremstyle{definition}

\newtheorem{pb}[thm]{Problem}
\newtheorem{rem}[thm]{Remark}}
\newcommand{\R}{\mathbb{R}}

\newcommand{\RP}{\mathbb{R}P}
\newcommand{\Z}{\mathbb{Z}}

\newcommand{\CP}{\mathbb{C}P}

\newcommand{\C}{\mathbb{C}}

\renewcommand{\O}{\mathcal{O}}

\renewcommand{\epsilon}{\varepsilon}
\renewcommand{\P}{\mathcal{P}}

\newcommand{\Area}{{\operatorname{Area}}}
\newcommand{\Tor}{{\operatorname{Tor}}}

\newcommand{\g}{{\mathfrak{g}}}
\newcommand{\Sing}{\mathop{\mathrm{Sing}}}

\title[Real algebraic curves with
large finite number of real points]{Real algebraic curves with
large finite number \\ of real points}

\author{Erwan Brugall\'{e}}
\address{Erwan Brugall\'{e}, Universit\'{e} de Nantes, Laboratoire de
  Math\'{e}matiques Jean Leray, 2 rue de la Houssini\`{e}re, F-44322 Nantes Cedex 3,
France}
\email{erwan.brugalle@math.cnrs.fr}

\author{Alex Degtyarev}
\address{Alex Degtyarev,
Bilkent University\\
Department of Mathematics\\
06800 Ankara, Turkey}
\email{degt@fen.bilkent.edu.tr}

\author{Ilia Itenberg}
\address{Ilia Itenberg,
Institut de Math\'{e}matiques de Jussieu--Paris Rive Gauche\\
Sorbonne Universit\'{e}\\
4 place Jussieu, 75252 Paris Cedex 5, France \\
and
D\'{e}partement de Math\'{e}mati\-ques et Applications,
Ecole Normale Sup\'{e}rieure \\
45 rue d'Ulm, 75230 Paris Cedex 5, France}
\email{ilia.itenberg@imj-prg.fr}

\author{Fr\'{e}d\'{e}ric Mangolte}
\address{Fr\'{e}d\'{e}ric Mangolte, Laboratoire angevin de recherche en math\'{e}matiques (LAREMA),
Universit\'{e} d\textquoteright{}Angers, CNRS, 49045 Angers Cedex 01, France}
\email{frederic.mangolte@univ-angers.fr}
\urladdr{http://www.math.univ-angers.fr/~mangolte}

\begin{document}

\begin{abstract}
We address the problem of the maximal finite number of real points of a real algebraic curve
(of a given degree and, sometimes, genus) in the projective plane.
We improve the known upper and lower bounds and construct close to optimal
curves of small degree. Our upper bound is sharp if the genus is small as compared
to the degree. Some of the results are extended to other real algebraic surfaces,
most notably ruled.
\end{abstract}

\maketitle

\section{Introduction}\label{sec:intro}

A \emph{real algebraic variety}  $(X,c)$ is a complex algebraic variety
equipped with an anti-holomorphic involution $c\colon X \to X$,
called a \emph{real structure}.
We denote by $\R X$ the real part of $X$, {\it i.e.},
the fixed point set of $c$.
With a certain abuse of language, a real algebraic variety is called
\emph{finite} if so is its real part.
Note that each real point of a finite real algebraic variety of positive dimension
is in the singular locus of the variety.

\subsection{Statement of the problem}\label{section-statement}

In this paper we mainly deal with the first non-trivial case,
namely, finite real algebraic curves in $\C P^2$.
 (Some of the results are extended to more general surfaces.)
The degree of such a curve $C \subset \C P^2$ is necessarily even, $\deg C = 2k$.
Our primary concern is the number $|\R C|$ of real points of $C$.

\begin{pb}
For a given integer $k\ge 1$, what is the
maximal number
$$
\delta(k) = \max\{|\R C| : \text{$C \subset \CP^2$ a finite real algebraic curve, $\deg C = 2k$} \}?
$$
For given integers $k\ge 1$ and $g\ge 0$, what is the maximal number
$$
\delta_g(k) = \max\{|\R C| : \text{$C \subset \CP^2$ a finite real algebraic curve of genus $g$, $\deg C = 2k$} \}?
$$
(See Section \ref{Comessatti_inequalities} for our convention for the genus of reducible curves.)
\end{pb}

The Petrovsky inequalities
(see \cite{Petrovsky} and Remark \ref{Petrovsky})
result
in
the following upper bound:
$$|\R C|\le \frac{3}{2}k(k-1) + 1. $$
Currently, this bound is the best known. Furthermore, being
of topological nature, it is sharp in the realm of pseudo-holomorphic curves.
Indeed,
consider a rational simple Harnack curve
of degree $2k$ in $\C P^2$ (see \cite{Mik11, KenOko06,Bru14b});
this curve has $(k - 1)(2k - 1)$ solitary real nodes
(as usual, by a node we mean a non-degenerate 
double point, {\it i.e.}, an $A_1$-singularity)
and an oval (see Remark \ref{Petrovsky} for the definition) surrounding $\frac{1}{2}(k - 1)(k - 2)$ 
of them. 
One can erase all inner 
nodes,
leaving the oval empty.
Then, \emph{in the pseudo-holomorphic category}, the oval can be contracted
to an extra 
solitary 
node, giving rise to a finite real pseudo-holomorphic
curve $C \subset \C P^2$ of degree $2k$ with $|\R C| = \frac{3}{2}k(k-1) + 1$.

\medskip

\subsection{Principal results}\label{section-results}

For the moment, the exact value of $\delta(k)$ is known only for
$k\le 4$.
The upper (Petrovsky inequality) and lower bounds
for a few
small values of $k$
are
as follows:
$$\begin{array}{c | c| c| c| c| c| c| c| c| c| c}
  k & 1& 2& 3& 4& 5& 6& 7& 8 & 9 &10
  \\\hline \delta(k)\le & 1&  4& 10& 19& 31 & 46& 64 &85 & 109& 136
\\\hline  \delta(k)\ge & 1& 4 & 10 & 19 & 30 & 45 &59  & 78 &  98   & 123
\end{array}$$
The cases $k = 1, 2$ are obvious
(union of two complex
conjugate lines or conics, respectively).
The lower bound for $k= 6$ is given by Proposition 
\ref{prop:better 12}, and
all other cases are covered by Theorem \ref{thm:exact}.
Asymptotically, we have
$$
\frac{4}{3}k^2 \lesssim \delta(k) \lesssim \frac{3}{2}k^2,
$$
where the lower bound follows from Theorem \ref{thm:exact}.

A finite real sextic $C_6$ with $|\R C_6| = \delta(3) =10$
was constructed by D. Hilbert \cite{Hilbert}.
A finite real octic $C_8$ with $|\R C_8| = \delta(4) = 19$
could easily
be
obtained by perturbing a quartuple conic,
although
we could not find such an octic in the literature.
The best previously known asymptotic lower bound
$\delta(k) \gtrsim \frac{10}{9}k^2$
is found in M. D. Choi, T. Y. Lam, B. Reznick \cite{CLR}.

With the genus $g = g(C)$ fixed, the upper bound
$$\delta_g(k) \leq k^2 + g + 1$$
is also given by a strengthening 
of the Petrovsky inequalities 
(see Theorem \ref{thm:upper con}).
In Theorem \ref{thm:cp2 low}, we show that this bound is sharp
for $g \leq k - 3$.

Most results extend to curves in ruled surfaces:
upper bounds are given by Theorem \ref{thm:upper con} (for $g$ fixed)
and Corollary \ref{cor:pet}; an asymptotic lower bound is given by Theorem \ref{thm:toric}
(which also covers arbitrary projective toric surfaces), and a few sporadic constructions
are discussed in Sections \ref{sec:hirz}, \ref{ellipsoid}.



\subsection{Contents of the paper}\label{contents}
In Section \ref{Comessatti_inequalities}, we obtain the upper bounds,
derived essentially from the Comessatti inequalities. In Section \ref{tools},
we discuss the auxiliary tools used in the constructions, namely,
the patchworking techniques, bigonal curves and \emph{dessins d'enfants},
and deformation to the normal cone. Section \ref{sec:plane} is dedicated
to  curves in $\C P^2$: we recast the upper bounds, describe a general construction
for toric surfaces (Theorem \ref{thm:toric}) and a slight improvement for the projective plane
(Theorem \ref{thm:exact}), and prove the sharpness of the bound $\delta_g(k) \leq k^2 + g + 1$
for curves of small genus.
In Section \ref{sec:hirz}, we consider surfaces ruled over $\R$,
proving the sharpness of the upper bounds for small bi-degrees and for small genera.
Finally, Section \ref{ellipsoid} deals with finite real curves in the ellipsoid.

\subsection{Acknowledgments}
Part of the work on this project was accomplished during
the second and third authors' stay at the \emph{Max-Planck-Institut f\"{u}r Mathematik}, Bonn.
We are grateful to the MPIM and its friendly staff for
their hospitality and excellent working conditions.
We extend our gratitude to
Boris Shapiro,
who brought the finite real curve problem to our attention and supported our work
by numerous fruitful discussions.
We would also like to thank Ilya Tyomkin for his help in specializing  general
statements from \cite{ShuTyo06}
to a few specific situations.

\section{Strengthened Comessatti inequalities}\label{Comessatti_inequalities}


Let $(X,c)$ be a smooth
real projective surface.
We denote by $\sigma^\pm_{\text{\rm inv}}(X,c)$ (respectively,
$\sigma^\pm_{\text{\rm skew}}(X,c)$)
the inertia indices of the invariant (respectively, skew-invariant) sublattice
of the involution  $\ $
$c_*\colon H_2(X; \Z) \to H_2(X; \Z)$ induced by $c$.
The following statement is standard.

\begin{prop}[see, for example, \cite{Wil}]\label{calculation}
One has
$$
\sigma^-_{\text{\rm inv}}(X,c) = \frac{1}{2}(h^{1, 1}(X) + \chi(\R X)) - 1, \qquad
\sigma^-_{\text{\rm skew}}(X,c) = \frac{1}{2}(h^{1, 1}(X) - \chi(\R X)),
$$
where $h^{\bullet, \bullet}$ are the Hodge numbers and $\chi$ is the topological Euler characteristic.
\end {prop}

\begin{cor}[Comessatti inequalities]\label{Comessatti}
One has
$$
2 - h^{1, 1}(X) \leq \chi(\R X) \leq h^{1, 1}(X).
$$
\end{cor}

\begin{rem}\label{Petrovsky}
Let $C \subset \CP^2$ be a smooth real curve of degree $2k$.
Recall that an \emph{oval} of $C$ is a connected component $\oval \subset \R C$
bounding a disk in $\R P^2$; the latter disk is called the \emph{interior} of~$\oval$.
An oval $\oval$ of $C$ is called
{\it even} (respectively, {\it odd})
if $\oval$ is contained inside an even (respectively, odd) number of other ovals of $C$;
the number of even (respectively, odd) ovals
of a given curve~$C$ is denoted by~$p$ (respectively,~$n$).
The classical Petrovsky inequalities \cite{Petrovsky} state that
$$
p-n \leq \frac{3}{2}k(k - 1) + 1, \qquad n-p \leq \frac{3}{2}k(k - 1).
$$
These inequalities can be obtained by applying Corollary \ref{Comessatti}
to the double covering of $\CP^2$
branched along
$C \subset \CP^2$
(see \textit{e.g.} \cite{Wil}, \cite[Th. 3.3.14]{ma-book}). 
\end{rem}

The Comessatti and Petrovsky inequalities, strengthened in several ways (see, {\it e.g.}, \cite{Viro86}),
have a variety of applications. For example, for nodal
finite real
rational curves in $\CP^2$ we immediately obtain the following statement.

\begin{prop}\label{upper-bound-rational}
Let $C \subset \CP^2$ be a nodal finite rational curve of degree $2k$.
Then, $|\R C| \leq k^2 + 1$.
\end{prop}

\begin{proof}
Denote by $r$ the number of real 
nodes of $C$,
and denote by $s$ the number of pairs of complex conjugate 
nodes of $C$.
We have $r + 2s = (k - 1)(2k - 1)$.
Let, further, $Y$ be the double covering of $\C P^2$ branched along the smooth real curve $C_t \subset \CP^2$
obtained from $C$ by a small perturbation creating an oval from each real node of $C$.
The union of $r$ small discs bounded by $\R C_t$ is denoted by $\RP^2_+$;
let $\bar c \colon Y \to Y$ be the lift of the real structure
such that the real part projects onto $\R P^2_+$.
Each pair of complex conjugate nodes of $C$
gives rise to a pair of $\bar c_*$-conjugate vanishing cycles in $H_2(Y; \Z)$;
their difference is a skew-invariant class of square $-4$,
and the $s$ square $-4$ classes thus obtained are pairwise orthogonal.

Since $h^{1, 1}(Y) = 3k^2 - 3k + 2$ (see, {\it e.g.} \cite{Wil}), 
Corollary \ref{Comessatti}
implies that
$$\chi(\R Y) \leq h^{1, 1}(Y) - 2s = 3k(k - 1) + 2 - 2s = k^2 + 1 + r.$$
Thus, $r \leq k^2 + 1$.
\end{proof}

The above statement can be generalized to the case of
not necessarily nodal curves of arbitrary genus
in any smooth real projective surface.

Recall that the \emph{geometric genus} $g(C)$ of an irreducible and  reduced
algebraic curve $C$ is the genus of its normalization. If $C$ is
reduced with irreducible components $C_1,\ldots, C_n$, the geometric
genus of $C$ is defined by
$$g(C)=g(C_1)+\ldots +g(C_n) +1-n. $$
In other words, $2-2g(C)=\chi(\tilde C)$, where $\tilde C$ is the
normalization.

%

Define also the \emph{weight} $\wt{p}$ of a solitary point~$p$ of a
real curve~$C$ as the minimal number of blow-ups \emph{at real points}
necessary to resolve~$p$. More precisely, $\wt{p}=1+\sum\wt{p_i}$, the
summation running over all real points~$p_i$ over~$p$
of the strict transform of~$C$ blown
up at~$p$. For example, the weight of a simple node equals~$1$, whereas the
weight of an $A_{2n-1}$-type point equals~$n$.
If $|\R C|<\infty$, we define the \emph{weighted point count} $\wt{\R C}$ as
the sum of the weights of all real points of~$C$.

The topology of the ambient complex surface $X$ is present in the next statement
in the form of the coefficient
$$
T_{2,1}(X)=\frac{1}{6}\bigl(c_1^2(X) -5c_2(X)\bigr) = \frac{1}{2}\bigl(\sigma(X) - \chi(X)\bigr)
= b_1(X) - h^{1,1}(X)
$$
of the Todd genus (see \cite{Hir}).

\begin{thm}\label{thm:upper con} 
Let $(X,c)$ be a simply connected smooth real projective surface with
non-empty connected
real part. Let $C \subset X$ be an ample
reduced finite real algebraic curve
such that $[C]=2e$ in $H_2(X;\Z)$.
Then, we have
\begin{equation}\label{eq:upper}
|\R C|\le \wt{\R C} \le e^2 + g(C)
- T_{2,1}(X) +\chi(\R X) -1.
\end{equation}
Furthermore,
the inequality is strict unless all singular points of~$C$ are double.
\end{thm}

\begin{proof}
Since $[C] \in H_2(X; \Z)$ is divisible by~$2$, there exists
a real double covering $\rho\colon(Y, \bar c)\to(X,c)$ ramified at~$C$
and such that $\rho(\R Y)=\R X$.
By the embedded resolution of
singularities,
we can find a sequence of real blow-ups
$\pi_i\colon X_i\to X_{i-1}$, $i=1,\ldots,n$, real curves
$C_i=\pi^*C_{i-1}\bmod2\subset X_i$, and real double coverings
$\rho_i=\pi_i^*\rho_{i-1}\colon Y_i\to X_i$ ramified at~$C_i$
such that the curve $C_n$ and surface~$Y_n$ are nonsingular.
(Here, a \emph{real blow-up} is either a blow-up at a
real point or a pair of blow-ups at two conjugate points. By
$\pi_i^*C_{i-1}\bmod2$ we mean the reduced divisor obtained by retaining the
\emph{odd multiplicity components} of the divisorial pull-back
$\pi_i^*C_{i-1}$.)

Using Proposition~\ref{calculation},
we can rewrite~\eqref{eq:upper} in the form
$$
e^2 + g(C)
+
h^{1,1}(X) - T_{2,1}(X) +2\chi(\R X) -2\sigma^-_{\text{\rm inv}}(X,c) -  \wt{\R C} \geq 3.
$$
We proceed by induction and prove a modified version
of the latter inequality,
namely,
\begin{equation}\label{eq:upper'}
 e_i^2 + g(C_i) +
h^{1,1}(X_i) - T_{2,1}(X_i) + b_1^{-1}(Y_i)
+2\chi(\R X_i) -2\sigma^-_{\text{\rm inv}}(X_i,c) - \wt{\R C_i}\ge3,
\end{equation}
where $[C_i]=2e_i\in H_2(X_i;\Z)$ and $b_1^{-1}(\cdot)$ is the dimension of the
$(-1)$-eigenspace of~$\rho_*$ on $H_1(\cdot;\C)$.

For the ``complex'' ingredients of~\eqref{eq:upper'}, it suffices to consider
a blow-up $\pi\colon\tilde X\to X$ at a singular point~$p$ of~$C$, not
necessarily real, of multiplicity~$O\ge2$. Denoting by~$C'$ the strict transform
of~$C$, we have $\tilde C=\pi^*C\bmod2=C'+\epsilon E$,
where $E=\pi^{-1}(p)$ is the exceptional divisor and $O=2m+\epsilon$,
$m\in\Z$, $\epsilon=0,1$. Then, in 
obvious notation,
\begin{equation*}
e^2=\tilde e^2+m^2,\qquad
g(C)=g(\tilde C)+\epsilon,\qquad
h^{1, 1}(X) = h^{1, 1}(\tilde X) - 1, \qquad
T_{2, 1}(X) = T_{2, 1}(\tilde X) + 1.
\end{equation*}
Furthermore, from the
isomorphisms
$H_1(\tilde Y,\tilde\rho^*E)=H_1(Y,p)=H_1(Y)$ we easily conclude that
$$b_1^{-1}(Y)\ge b_1^-(\tilde Y)-b_1^-(\tilde\rho^*E)\ge b_1^{-1}(\tilde Y)-2(m-1).$$
It follows that, when passing from~$\tilde X$ to~$X$, the increment in the
first five terms of~\eqref{eq:upper'} is at least $(m-1)^2+\epsilon-1\ge-1$;
this increment equals~$(-1)$ if and only if $p$ is a double point of~$C$.

For the last three terms, assume first that the singular point~$p$ above is
real. Then
\begin{equation*}
  \chi(\R X)=\chi(\R\tilde X)+1,\qquad
  \sigma^-_{\text{\rm inv}}(X_i,c)=\sigma^-_{\text{\rm inv}}(\tilde X_i,\tilde c),\qquad
  \wt{\R C}=\wt{\R\tilde C}+1,
\end{equation*}
and the total increment in~\eqref{eq:upper'} is positive; it equals~$0$ if
and only if $p$ is a double point.

Now, let $\pi\colon\tilde X\to X$ be a pair of blow-ups at two complex conjugate
singular points of~$C$. Then
\begin{equation*}
  \chi(\R X)=\chi(\R\tilde X),\qquad
  \sigma^-_{\text{\rm inv}}(X_i,c)=\sigma^-_{\text{\rm inv}}(\tilde X_i,\tilde c)-1,\qquad
  \wt{\R C}=\wt{\R\tilde C},
\end{equation*}
and, again, the total increment is positive, equal to~$0$ if and only if both
points are double.

To establish~\eqref{eq:upper'} for the last, nonsingular,
curve~$C_n$,
we use the following observations:
\begin{itemize}
\item $\chi(Y_n)=2\chi(X_n)-\chi(C_n)$ (the Riemann-Hurwitz formula);
\item $\sigma(Y_n)=2\sigma(X_n)-2e_n^2$ (Hirzebruch's theorem);
\item $b_1(Y_n) - b_1(X_n) = b_1^{-1}(Y_n)$, as $b_1^{+1}(Y_n) = b_1(X_n)$ {\it via} the transfer map;
\item $\chi(\R Y_n)=2\chi(\R X_n)$, since
$\R C_n=\varnothing$ and $\R Y_n\to\R X_n$ is an unramified double covering.
\end{itemize}
Then, \eqref{eq:upper'} takes the form
\begin{equation*}
\sigma^-_{\text{\rm inv}}(Y_n, \bar c_n)\ge \sigma^-_{\text{\rm inv}}(X_n,c_n),
\end{equation*}
which is obvious in view of the transfer map
$H_2(X_n;\R)\to H_2(Y_n;\R)$: this map is equivariant and isometric up to a
factor of~$2$.

Thus, there remains to notice that
$b_1^{-1}(Y_0)=0$.
Indeed, since $C_0=C$ is assumed ample, $X \smallsetminus C$ has homotopy type of a CW-complex of dimension $2$
(as a Stein manifold). Hence, so does $Y \smallsetminus C$,
and the homomorphism $H_1(C; \R) \to H_1(Y; \R)$ is surjective. Clearly, $b_1^{-1}(C) = 0$.
\end{proof}



\begin{cor}\label{cor:pet}
 Let $(X,c)$ and $C\subset X$ be as in Theorem~\ref{thm:upper con}.
Then, we have
\begin{equation*}
  2|\R C| \le  3e^2  - e\cdot c_1(X) - T_{2,1}(X) +\chi(\R X),
 \end{equation*}
 the inequality being strict unless each singular
point of $C$ is a solitary real node of $\R C$.
\end{cor}
\begin{proof}
  By the adjunction formula we have
  $$g(C)\le 2e^2 -e\cdot c_1(X) +1 -|\R C|, $$
  and the result follows from Theorem \ref{thm:upper con}.
\end{proof}

\begin{rem}\label{rem:real_divisors}
The assumptions $\pi_1(X)=0$ and $b_0(\R X)=1$
in Theorem \ref{thm:upper con} are mainly used to assure the existence
of a real double covering $\rho\colon Y \to X$ ramified over a given real divisor $C$.
In general, one should speak about the divisibility by $2$ of the \emph{real divisor class}
$|C|_\R$, {\it i.e.}, class of real divisors modulo real linear equivalence.
(If $\R X \ne \varnothing$, one can alternatively speak about the set of real divisors
in the linear system $|C|$ or a real point of $\Pic(X)$.)
A necessary condition is the vanishing
$$
[C] = 0 \in H_{2n - 2}(X; \Z/2\Z), \qquad [\R C] = 0 \in H_{n - 1}(\R X; \Z/2\Z),
$$
where $n = \dim_\C(X)$ and $[\R C]$ is the homology class of the real part of (any representative of) $|C|$
(the sufficiency of this condition in some special cases is discussed in Lemma \ref{lem:real sqrt} below).
If not empty, the set of double coverings ramified over $C$ and admitting real structure
is a torsor over the space of $c^*$-invariant elements of $H^1(X; \Z/2\Z)$.
\end{rem}

The proof of the following theorem repeats literally that of Theorem \ref{thm:upper con}.

\begin{thm}\label{thm:upper}
Let $(X,c)$ be a smooth real projective surface
and $C\subset X$ an ample finite reduced real algebraic curve
such that the class $|C|_\R$ is divisible by $2$.
A choice of
a real
double covering $\rho\colon Y\to X$ ramified over~$C$
defines
a decomposition of $\R X$
into two disjoint subsets $\R X_+=\rho(\R Y)$ and $\R X_-$
consisting of whole components. Then, we have
$$  \wt{\R C\cap \R X_+}-\wt{\R C\cap \R X_-} \le
 e^2 + g(C) - T_{2,1}(X) +\chi(\R X_+) -\chi(\R X_-) -1, $$
the inequality being strict unless all singular points of~$C$ are double. \qed
\end{thm}

\section{Construction tools}\label{tools}

\subsection{Patchworking}\label{partchworking}

If $\Delta$ is a convex
lattice polygon
contained in the non-negative quadrant $(\R_{\geq 0})^2 \subset \R^2$,
we denote by $\Tor(\Delta)$ the toric
variety associated with $\Delta$; this variety is a surface if $\Delta$ is non-degenerate.
In the latter case, the complex torus $(\C^*)^2$ is naturally embedded in $\Tor(\Delta)$.
Let $V \subset (\R_{\geq 0})^2 \cap \Z^2$ be a finite set,
and let $P(x, y) = \sum_{(i, j) \in V}a_{ij}x^iy^j$ be a real polynomial in two variables.
The {\it Newton polygon} $\Delta_P$ of $P$ is the convex hull in $\R^2$ of those points in $V$
that correspond to the non-zero monomials of $P$.
The polynomial $P$ defines an algebraic curve in the $2$-dimensional complex torus $(\C^*)^2$;
the closure of this curve in $\Tor(\Delta_P)$ is an algebraic curve $C \subset \Tor(\Delta_P)$.
If $Q$ is a quadrant of $(\R^*)^2 \subset (\C^*)^2$ and $(a, b)$ is a vector in $\Z^2$,
we denote by $Q(a, b)$ the quadrant
$$\{(x, y) \in (\R^*)^2 \ | \ ((-1)^a x, (-1)^b y) \in Q\}.$$
If $e$ is an integral segment whose direction is generated by a primitive
integral vector $(a,b)$, we abbreviate $Q(e^\perp):=Q(b,-a)$. 
A real algebraic curve $C\subset\Tor(\Delta)$ is said to be {\it $\frac{1}{4}$-finite}
(respectively, {\it $\frac{1}{2}$-finite})
if the intersection of
the real part
$\R C$ with the 
positive quadrant $(\R_{>0})^2$
(respectively, the union
$(\R_{>0})^2 \cup (\R_{>0})^2(1, 0)$
is finite. 



Fix a subdivision ${\mathcal S} = \{\Delta_1,\ldots,\Delta_N\}$ 
of a convex polygon $\Delta \subset (\R_{\geq 0})^2$
such that there exists a piecewise-linear convex function $\nu\colon \Delta \to \R$
whose maximal linearity domains 
are precisely the non-degenerate lattice polygons $\Delta_1,\ldots,\Delta_N$. 
Let $a_{ij}$, $(i, j) \in \Delta \cap \Z^2$, be a collection of real numbers such that
$a_{ij} \ne 0$
whenever $(i, j)$ is a vertex of 
$\mathcal S$. 
This gives rise to $N$ real algebraic curves $C_k$, $k = 1,\ldots,N$:
each curve $C_k \subset \Tor(\Delta_k)$ is defined by the polynomial
$$P(x, y) = \sum_{(i, j) \in \Delta_k \cap \Z^2}a_{ij}x^iy^j$$
with the Newton polygon $\Delta_k$.

Commonly, we denote by $\Sing(C)$ the set of singular points of a
curve~$C$.
If $C\subset\Tor(\Delta)$ and $e\subset\Delta$ is an edge, we put
$T_e(C) := C \cap D(e)$, where $D(e)$ is the toric
divisor corresponding to~$e$.

Assume that each curve $C_k$ is nodal and 
$\Sing(C_k)$ is disjoint from 
the toric divisors of $\Tor(\Delta_k)$ (but $C_k$ can be tangent with arbitrary order
of tangency to some toric divisors).
For each
inner edge $e=\Delta_i\cap\Delta_j$ of $\mathcal S$,
the toric divisors corresponding to $e$ in $\Tor(\Delta_i)$ and $\Tor(\Delta_{j})$
are naturally identified, as they
both are $\Tor(e)$.
The intersection points of $C_i$ and $C_{j}$ with these toric divisors are
also identified,
and,
at each such point $p\in\Tor(e)$, the orders of intersection
of $C_i$ and $C_{j}$ with
$\Tor(e)$
automatically coincide; this common order is denoted by $\mult p$ and,
if $\mult p>1$, the point $p$
is called {\it fat}.
Assume that
$\mult p$ is even for each \emph{fat} point~$p$
and that the local branches of $C_i$ and $C_{j}$
at
each \emph{real} fat point~$p$
are in the same quadrant $Q_p\subset(\R^*)^2$.

Each edge $E$ of~$\Delta$ is a union of exterior edges $e$ of~$\mathcal S$;
denote the set of 
these edges by~$\{E\}$ and, given $e\in\{E\}$, let $k(e)$ be
the index such that $e\subset\Delta_{k(e)}$.
The toric divisor $D(E) \subset \Tor(\Delta)$
is a smooth real rational curve whose real part $\R D(E)$ is divided into two halves
$\R D_\pm(E)$ 
by the intersections with other toric divisors of
$\Tor(\Delta)$;
we denote by $\R D_+(E)$ the half
adjacent to the positive quadrant of $(\R^*)^2$. 
Similarly, the toric divisor $D(e) \subset \Tor(\Delta_{k(e)})$ 
is divided into
$\R D_\pm(e)$. 


\begin{thm}[Patchworking construction; essentially, Theorem 2.4 in \cite{Sh}]\label{patchworking} 
Under the assumptions above,
there exists a family of real polynomials $P^{(t)}(x, y)$, $t\in\R_{>0}$,
with the Newton polygon $\Delta$,
such that, for sufficiently small~$t$, the curve $C^{(t)} \subset \Tor(\Delta)$ defined by $P^{(t)}$
has the following properties: 
\begin{itemize}
\item the curve $C^{(t)}$ is nodal and
$\Sing(C^{(t)})$ is disjoint from 
the toric divisors;
\item if all curves $C_1,\ldots,C_N$ are $\frac{1}{2}$-finite {\rm (}respectively, $\frac{1}{4}$-finite{\rm )},
then
so is
$C^{(t)}$;
\item there is an injective map
$$\Phi\colon \coprod_{k = 1}^N \Sing(C_k) \to \Sing(C^{(t)}),$$
such that the image
of each real point
is a real point
of
the same type {\rm (}solitary/non-solitary{\rm )}
and
in the same quadrant of $(\R^*)^2$, and the image
of each imaginary point
is imaginary;
\item 
there is a partition
$$\Sing(C^{(t)}) \sminus \text{\rm image of $\Phi$}=\coprod_p\Pi_p,$$
$p$ running over all fat points, so that $|\Pi_p|=2m-1$ if $\mult p=2m$.
The points in $\Pi_p$ are imaginary
if $p$ is imaginary and real and solitary if $p$ is real; in the latter case,
$(m - 1)$ of these points lie
in~$Q_p$ and the 
others $m$ points 
lie in
$Q_p(e_p^\perp)$, where $p\in\Tor(e_p)$; 
\item for each edge $E$ of $\Delta$, there is a bijective map
$$\Psi_E\colon \coprod_{e \in\{E\}} T_e(C_{k(e)}) \to T_E(C^{(t)})$$
preserving the intersection multiplicity and the position of points in
$\R D_\pm(\cdot)$ or $D(\cdot)\sminus\R D(\cdot)$.
\end{itemize}
\end{thm}

\begin{proof}
To
deduce the statement
from \cite[Theorem 2.4]{Sh},
one can use Lemma 5.4(ii) in \cite{Sh05}
and the deformation patterns described in \cite{IKS}, Lemmas 3.10 and 3.11
(\emph{cf}.\space also the curves $C_{*, 0, 0}$ in Lemma \ref{lem:deg2} below).
\end{proof}

%
%
%

\subsection{Bigonal curves \emph{via} \emph{dessins d'enfants}}\label{dessins}

We denote by $\Sigma_n$, $n\ge 0$, the Hirzebruch surface of degree $n$, {\it i.e.}, 
$\Sigma_n=\mathbb P(\mathcal O_{\CP^1}(n)\oplus \mathcal O_{\CP^1})$.
Recall that $\Sigma_0=\CP^1\times\CP^1$ and $\Sigma_1$ is the
blow-up of $\CP^2$ at a point.
The
bundle projection
induces a map $\pi\colon\Sigma_n \to \CP^1$, and we denote by $F$ a fiber
of $\pi$; it is isomorphic to $\CP^1$.
The images of $\mathcal O_{\CP^1}$ and  $\mathcal O_{\CP^1}(n)$ are denoted
by $B_0$ and $B_\infty$, respectively; these curves are sections of~$\pi$.
The group $H^2(\Sigma_n;\C)=H^{1,1}(X;\C)$ is generated by
the classes of $B_0$ and $F$, and we have
$$
[B_0]^2=n,\quad
[B_\infty]^2=-n,\quad
[F]^2=0,\quad
B_\infty\sim B_0-nF,\quad
c_1(\Sigma_n)=2[B_0]+(2-n)[F].
$$
(If $n>0$, the \emph{exceptional section} $B_\infty$
is the only irreducible curve of negative self-intersection.)
In other words, we have $D\sim aB_0+bF$ for each divisor $D\subset\Sigma_n$,
and the pair $(a,b)\in\Z^2$ is called the \emph{bidegree} of~$D$.
The cone of effective divisors is generated by~$B_\infty$
and~$F$, and the cone of ample divisors is
$\{aB_0+bF\,|\,a,b>0\}$.

In this section, we equipp $\CP^1$ with the standard complex
conjugation, and the surface $\Sigma_n$ with the real structure $c$
induced by the standard complex conjugation on
$\mathcal O_{\CP^1}(n)$.
Unless $n=0$, this is the only real
structure on $\Sigma_n$ with nonempty real part.
In particular $c$ acts on $H^2(\Sigma_n;\C)$ as $-\operatorname{Id}$, and so
$\sigma^-_{\text{\rm inv}}(X, c)=0$.
The real part of $\Sigma_n$ is a torus if $n$ is even, and a Klein
bottle if $n$ is odd.
In the former case,
the complement $\R\Sigma_n\sminus(\R B_0\cup\R B_\infty)$
has two connected components, which we denote by $\R\Sigma_{n,\pm}$.

\begin{lemma}\label{lem:deg2}
Given integers $n>0$, $b\ge0$, and $0\le q\le n+b-1$,
there
exists a real algebraic rational curve $C=C_{n,b,q}$ in $\Sigma_{2n}$ of
bidegree $(2,2b)$
such that (see Figure \ref{fig:deg2}):
\begin{enumerate}
\item\label{deg2:1}
  all singular points of~$C$ are $2n+2b-1$ solitary nodes; $n+b+q$ of them
  lie
in
  $\R\Sigma_{2n,+}$, and the other $n+b-q-1$ lie
 in $\R\Sigma_{2n,-}$;
\item\label{deg2:2}
  the real part $\R C$ has a single extra oval~$\oval$, which is contained in
  $\R\Sigma_{2n,-}\cup \R B_0\cup\R B_\infty$ and does not contain any of the
  nodes in its interior;
\item\label{deg2:3}
  each intersection $p_\infty:=\oval\cap B_\infty$ and $p_0:=\oval\cap B_0$ consists
  of a single point, the multiplicity being $2b$ and $4n+2b-2q$,
  respectively; the points $p_0$ and $p_\infty$ are on the same fiber $F$.
\end{enumerate}
\begin{figure}[h!]
\begin{center}
\begin{tabular}{c}
  \includegraphics[width=6cm, angle=0]{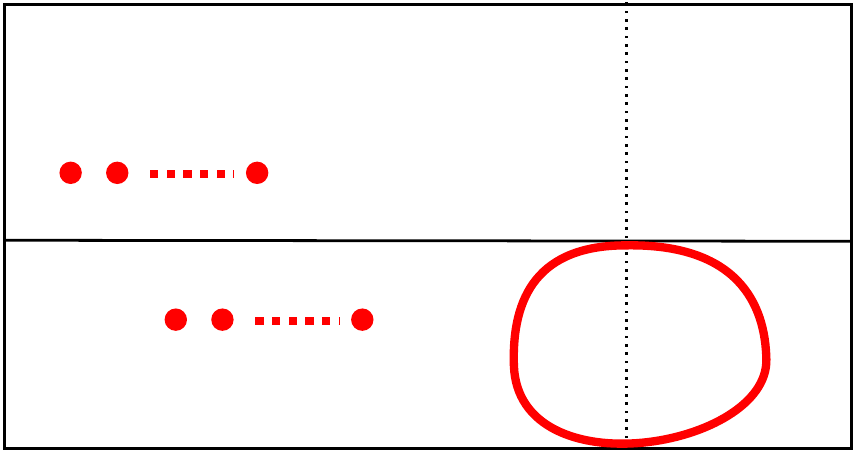}
 \put(-159,60){$\overbrace{\qquad \quad\ \ }^{n+b+q}$}
 \put(-195,40){$\R B_0$}
  \put(-195,0){$\R B_\infty$}
\put(-140,20){$\underbrace{\qquad \quad \ \ \ }_{ n+b-q -1}$}
 \put(-50,-10){$p_\infty$}
 \put(-40,50){$p_0$}
\\  \\  $\R C_{n,b,q}$
\end{tabular}
\end{center}
\caption{}
\label{fig:deg2}
\end{figure}
This curve can be perturbed to a curve
$\widetilde C_{n,b,q}\subset\Sigma_{2n}$ satisfying conditions~(\ref{deg2:1})
and~(\ref{deg2:2}) and the following modified version of
condition~(\ref{deg2:3}):
\begin{enumerate}
\item[(\~3)]
the oval~$\oval$ intersects $B_\infty$ and $B_0$ at,
  respectively, $b$ and $2n+b-q$ simple tangency points.
\end{enumerate}
\end{lemma}

Note that $C_{n,b,q}$ intersects $B_0$ in $q$ additional pairs of complex conjugate
points.


\begin{proof}
Up to elementary transformations of $\Sigma_{2n}$ (blowing up the point of
intersection $C\cap B_\infty$ and blowing down the strict transforms of the
corresponding fibers)
we may assume that $b=0$
and, hence, $C$ is disjoint from $B_\infty$.
Then, $C$ is given by $P(x,y)=0$, where
\begin{equation}
 P(x,y)=y^2 + a_1(x)y + a_2(x) , \qquad \deg a_i(x)= 2in.
\label{eq:ai}
\end{equation}
(Strictly speaking, $a_i$ are sections of appropriate line bundles, but we
pass to affine coordinates and regard $a_i$ as polynomials.)
We will construct the curves using the techniques of \emph{dessins
d'enfants}, cf. \cite{Orevkov:Riemann,DIK:elliptic,degt:book}.
Consider the
rational function $f\colon\CP^1\to\CP^1$ given by
$$f(x)=\frac{a_1^2(x)-4a_2(x)}{a_1^2(x)}.
$$
(This function differs from the $j$-invariant of the trigonal curve $C+B_0$ by
a few irrelevant factors.)
The \emph{dessin} of~$C$ is the graph $\mathcal D:=f^{-1}(\RP^1)$ decorated
as shown in Figure~\ref{fig:decoration}.
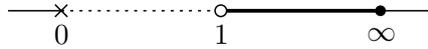
\begin{figure}[h!]
\centerline{\input{Figures/03-j.inc}}
\caption{Decoration of a dessin}\label{fig:decoration}
\end{figure}
In addition to $\times$-, $\circ$-, and $\bullet$-vertices, it may also have
\emph{monochrome} vertices, which are the pull-backs of the real critical
values of~$f$ other that $0$, $1$, or~$\infty$. This graph is real, and we
depict only its projection to the disk $D:=\CP^1/x\sim\bar x$, showing the
boundary $\partial D$ by a wide grey curve: this boundary corresponds to the
real parts $\R C\subset\R\Sigma_{2n}\to\RP^1$. Assuming that $a_1$, $a_2$
have no common roots, the real special vertices and edges of $\mathcal D$
have the following geometric interpretation:
\begin{itemize}
\item
a $\times$-vertex~$x_0$ corresponds to a double root of the polynomial
  $P(x_0,y)$; the curve is tangent to a fiber if $\val x_0=2$ and has a double
  point of type $A_{p-1}$, $p=\frac12\val x_0$, otherwise;
\item
a $\circ$-vertex~$x_0$ corresponds to an intersection $\R C\cap\R B_0$ of
  multiplicity $\frac12\val x_0$;
\item
the real part $\R C$ is empty over each point of a solid edge and consists of
  two points over each point of any other edge;
\item
the points of~$\R C$ over two $\times$-vertices~$x_1$, $x_2$ are in the same
half $\R\Sigma_{2n,\pm}$ if and only if one has $\sum\val z_i=0\bmod8$, the
summation running over all $\bullet$-vertices~$z_i$ in any of the two arcs of
$\partial D$ bounded by $x_1$, $x_2$.
\end{itemize}
(For the last item, observe that the valency of each $\bullet$-vertex is
$0\bmod4$ and the sum of all valencies equals $2\deg f=8n$;
hence, the sum in the statement is independent of the choice of the arc.)

Now, to construct the curves in the statements, we start with the dessin
$\widetilde{\mathcal D}_{n,0,0}$ shown in Figure~\ref{fig:dessin}, left: it
has $2n$ $\bullet$-vertices, $2n$ $\circ$-vertices, and $(2n+1)$
$\times$-vertices, two bivalent and $(2n-1)$ four-valent, numbered
consecutively along $\partial D$.
\begin{figure}[h!]
\centerline{\input{Figures/curve.inc}}
\caption{The dessin $\widetilde{\mathcal D}_{n,0,0}$ and its modifications}\label{fig:dessin}
\end{figure}
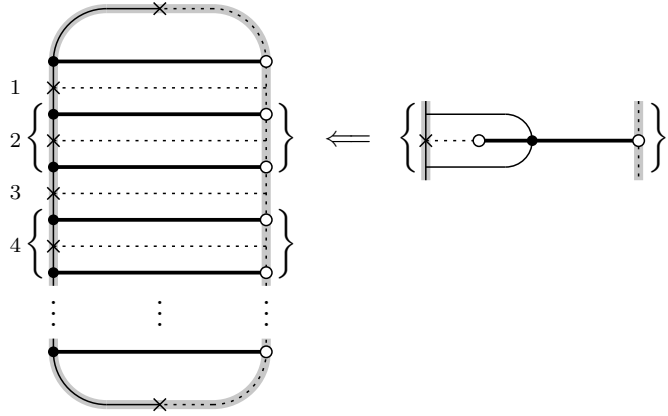
To obtain $\widetilde{\mathcal D}_{n,0,q}$, we replace $q$ disjoint embraced
fragments with copies of the fragment shown in Figure~\ref{fig:dessin},
right; by choosing the fragments replaced around \emph{even-numbered}
$\times$-vertices, we ensure that the solitary 
nodes would migrate
from $\R\Sigma_{2n,-}$ to $\R\Sigma_{2n,+}$. Finally,
$\mathcal D_{n,0,q}$ is obtained from
\smash{$\widetilde{\mathcal D}_{n,0,q}$} by
contracting the dotted real segments connecting the real $\circ$-vertices, so
that the said vertices collide to a single $(8n-4q)$-valent one.
Each of these dessins~$\mathcal D$ gives rise to a (not unique) equivariant
topological branched covering $f\colon S^2\to\CP^1$
(cf. \cite{Orevkov:Riemann,DIK:elliptic,degt:book}), and the Riemann
existence theorem gives us an analytic structure on the sphere~$S^2$
making~$f$ a real rational function $\CP^1\to\CP^1$. There remains to take
for~$a_1$ a real polynomial with a simple zero at each (double) pole of~$f$
and let $a_2:=\frac14a_1^2(1-f)$.
\end{proof}

Generalizing, one can consider a geometrically ruled surface
$\pi\colon\Sigma_n(\O):=\mathbb P(\O\oplus\mathcal O_B)\to B$,
where $B$ is a smooth compact real curve of genus $\g\ge1$ and $\O$ is a line
bundle, $\deg\O=n\ge0$.
If
$\O$ is also real, the surface
$\Sigma_n(\O)$ acquires a real structure;
the sections $B_0$ and $B_\infty$ are also real and we can speak about
 $\R B_0$, $\R B_\infty$.
The real line bundle $\O$ is said to be \emph{even} if
the $GL(1,\R)$-bundle $\R\O$ over $\R B$ is trivial
(cf. Remark \ref{rem:real_divisors}).
In this case, the real part $\R \Sigma_n(\O)$
is a disjoint union of
tori,
one torus~$T_i$
over each
real component $\R_iB$ of $B$, and each complement
$T_i^\circ:=T_i\smallsetminus \left(\R B_0\cup \R B_\infty\right)$ is made of two
connected components (open annuli).

A smooth compact real curve $B$ of genus $\g$ is called \emph{maximal}
if it has the maximal possible number of real connected
components: $b_0(\R B) = \g + 1$.

\begin{lemma}\label{arbitrary_genus_curve}
Let $n$, $\g$ be two integers, $n\ge \g-1\ge0$. Then
there exists an even real line bundle $\O$ of degree $\deg\O=2n$
over a maximal real
algebraic curve $B$ of genus $\g$,
and a nodal real algebraic
curve $C_{n}(\g)\subset\Sigma_{2n}(\O)$
realizing the class $2[B_0]\in H_2(\Sigma_{2n}(\O);\Z)$ such that
\begin{enumerate}
  \item $\R C_{n}(\g)\cap T_1$ consists of $2n$ solitary nodes, all
    in the same connected component of $T_1^\circ$;
  \item $\R C_{n}(\g)\cap T_2$ is a smooth
    connected curve, contained in a single connected component of
    $T_2^\circ$ except for $n$ real points of simple tangency of $C_{n}$
    and~$B_0$;
  \item $\R C_{n}(\g)\cap T_i$, $i\ge 3$, is a smooth
    connected curve,
    contained in a single connected component of
    $T_2^\circ$ except for one real point of simple tangency of $C_{n}$
    and~$B_0$.

\end{enumerate}
\end{lemma}

Note that we can only assert the existence of a ruled surface
$\Sigma_{2n}(\O)$: the analytic structure on~$B$ and line
bundle~$\O$ are given by the construction and cannot be fixed in
advance.

\begin{proof}
We proceed as in the proof of Lemma~\ref{lem:deg2}, with the
``polynomials''~$a_i$ sections of $\O^{\otimes i}$ in~\eqref{eq:ai}
and half-dessin
\smash{$\mathcal D_{n}(\g)/c_B$} in the surface
$D:=B/c_B$, which, in the case of maximal~$B$, is a disk with $\g$ holes;
as above, we have $\partial D=\R B$. The following
technical requirements are necessary and sufficient for the existence of a
topological ramified covering $f\colon B\to\CP^1$
(see \cite{DIK:elliptic,degt:book}) with $B$ the orientable double of~$D$:
\begin{itemize}
\item each \emph{region} (connected component of $D\sminus\mathcal D$)
  should admit an orientation inducing on the boundary the orientation inherited
  from $\RP^1$ (the order on~$\R$), and
\item each \emph{triangular} region ({\it i.e.}, one with a single vertex of each
  of the three
  special types $\times$, $\circ$, and~$\bullet$ in the boundary)
  should be a topological disk.
\end{itemize}
(For example, in the dessins \smash{$\widetilde{\mathcal D}_{n,0,q}$} in
  Figure~\ref{fig:dessin} the orientations are given by a chessboard coloring
  and all regions are triangles.)

The curve $C_n(\g)$ as in the statement is obtained from the dessin
$\mathcal{D}_n(\g)$ constructed as follows. If $\g=1$, then $\mathcal{D}_n(1)$
is the dessin in the annulus shown in Figure~\ref{fig:dessin_g}, left
\begin{figure}[h!]
\centerline{\input{Figures/curve_g.inc}}
\caption{The dessin $\mathcal D_n(1)$ and its modifications}\label{fig:dessin_g}
\end{figure}
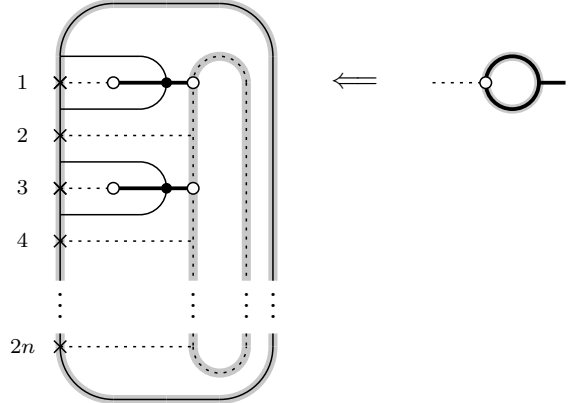
(which is a slight modification of $\widetilde{\mathcal{D}}_{n,0,n-1}$ in
Figure~\ref{fig:dessin}): it has $2n$ real four-valent $\times$-vertices,
$n$ inner four-valent $\bullet$-vertices, and $2n$ $\circ$-vertices, $n$ real
four-valent and $n$ inner bivalent. (Recall that each inner vertex in~$D$
doubles in~$B$, so that the total valency of the vertices of each kind sums
up to $8n=2\deg f$, as expected.)
This dessin is maximal in the sense that
all its regions are triangles.
To pass from $\mathcal{D}_n(1)$ to $\mathcal{D}_n(1+q)$, $q\le n$, we
replace small neighbourhoods of $q$ inner $\circ$-vertices with the fragments
shown in Figure~\ref{fig:dessin_g}, right, creating $q$ extra boundary
components.

Each dessin $\mathcal{D}_n(\g)$
satisfies the two conditions above and, thus, gives
rise to a ramified covering $f\colon B\to\CP^1$. The analytic structure
on~$B$ is given by the Riemann existence theorem, and $\O$ is the
line bundle $\mathcal O_B(\frac12P(f))$, where $P(f)$ is the divisor of poles
of~$f$. (All poles are even.)
Then, the curve in question is given by ``equation''~\eqref{eq:ai}, with the sections
$a_i\in H^0(B;\O^{\otimes i})$
almost determined by their zeroes:
$Z(a_1)=\frac12P(f)$ and $Z(a_2)=Z(1-f)$.
Further details of this construction
(in the more elaborate trigonal case)
can be found in \cite{DIK:elliptic,degt:book}.
\end{proof}

Next
few lemmas deal with the real lifts of the  curves constructed in
Lemma \ref{arbitrary_genus_curve} under a ramified double covering of
$\Sigma_{2n}(\O)$. First, we discuss the existence of
such coverings, \emph{cf}. Remark~\ref{rem:real_divisors}.

\begin{lemma}\label{lem:real sqrt}
Let $\Sigma_{n}(\O)$ be a real  ruled surface over a real algebraic curve
$B$ such that $\R B\ne\varnothing$, and let $D$ be a real divisor
on $X$. Then there exists a real divisor
$E$ on $X$ such that
$|D|_\R=2|E|_\R$ if and only if $[\R D]=0\in H_1(\R X;\Z/2\Z)$.
\end{lemma}
\begin{proof}
  By \cite[Proposition 2.3]{Har77}, we have
  \[
  \mbox{Pic}(\Sigma_{n}(\O)) \simeq \Z B_0 \oplus  \mbox{Pic}(B),
  \]
  and this isomorphism respects the action induced by the real
  structures.
  Let
\[
  |D|= m|B_0| + |D_0|.
\]
Then $m=[\R D]\circ[\R F]\bmod 2$, where $F$ is the fiber of the ruling over
a real point $p\in\R B$, and $D_0=D\circ B_\infty$, so that
$[\R D_0]=[\R D]\circ[\R B_\infty]$.
There remains to observe that $|D_0|_\R$ is divisible by~$2$ in $\R\Pic(B)$
if and only if $[\R D_0]=0\in H_0(B;\Z/2\Z)$.
  The ``only if'' part is clear, and the ``if'' part follows from
  the fact that $D_0$ can be deformed, through real divisors, to
  $(\deg D_0)p$.
\end{proof}


\begin{lemma}\label{lem:real double}
Let $X:=\Sigma_{n}(\O)$ be a real ruled surface over a real algebraic curve
$B$ such that $\R B\ne\varnothing$, and let $C$ be a reduced real divisor
on $X$ such that $[\R C]=0\in  H_1(\R X;\Z/2\Z)$. Then, for any surface
$S\subset\R X$ such that $\partial S=\R C$, there exists a real double
covering $Y\to X$ ramified over~$C$ such that $\R Y$ projects onto~$S$.
\end{lemma}
\begin{proof}
Pick one covering $Y_0\to X$, which exists by Lemma~\ref{lem:real sqrt}, and
let $S_0$ be the projection of~$\R Y_0$. We can assume that
$S_0\cap T_1=S\cap T_1$ for one of the components $T_1$ of $\R X$.
Given another component $T_i$, consider a path $\gamma_i$ connecting a
point in $T_i$ to one on $T_1$, and let
$\tilde\gamma_i=\gamma_i+c_*\gamma_i$;
in view of the
obvious equivariant isomorphism $H_1(Y;\Z/2\Z)\simeq H_1(B;\Z/2\Z)$,
these loops form a partial basis for
the space of $c_*$-invariant classes in $H_1(X;\Z/2\Z)$.
Now, it suffices to
twist $Y_0$ (\emph{cf}. Remark~\ref{rem:real_divisors})
by a cohomology class sending $\tilde\gamma_i$ to~$0$ or~$1$ if
$S\cap T_i$ coincides with $S_0\cap T_i$ or with the closure of its
complement, respectively.
%
%
%
\end{proof}

\begin{lemma}\label{prop:positive genus}
Let $n$, $\g$ be two integers, $n\ge \g-1\ge0$,
and let $B$, $\O$, and $C_n(\mathfrak g)\subset\Sigma_{2n}(\O)$ be as in
Lemma~\ref{arbitrary_genus_curve}.
Then
there exists a real double covering
$\Sigma_{n}(\O')\to\Sigma_{2n}(\O)$ ramified along
$B_0\cup B_\infty$ and such that the pullback
of $C_n(\g)$
 is a finite real algebraic curve $C'_n(\g)\subset\Sigma_{n}(\O')$
with
\[
\mathopen|\R C'_n(\g)\mathclose|=5n-1+\g.
\]
\end{lemma}
\begin{proof}
  By Lemma \ref{lem:real double}, there exists a real double covering
  $\Sigma_{n}(\O')\to\Sigma_{2n}(\O)$  ramified along the curve
$B_0\cup B_\infty$, such that the pull back in $\Sigma_{n}(\O')$
of the curve $C_n(\g)$ from Lemma \ref{arbitrary_genus_curve}
is a finite real algebraic curve $C'_n(\g)$.
Each node of $C_n(\g)$
gives
rise to two solitary real nodes of $C'_n(\g)$,
and each tangency point of
$C_n(\g)$ and $\R B_0$
gives rise to an extra solitary node of $C'_n(\g)$.
\end{proof}

\subsection{Deformation to the normal cone}\label{normal_cone}
We briefly recall the deformation to normal cone construction in the
setting we need here, and
refer for example to \cite{F} for more details.
Given $X$  a non-singular  algebraic surface, and $B\subset X$ a
non-singular  algebraic curve, we denote by $N_{B/X}$ the normal bundle of
$B$ in $X$,  its projective completion by
$E_{B}=\mathbb P(N_{B/X}\oplus \mathcal O_{B})$,
and we define  $B_\infty=E_{B}\sminus N_{B/X}$. Note that if both
$X$ and $B$ are real, then so are $E_B$ and $B_\infty$.

Let  $\mathcal X$
be the blow up of
$X\times \C$ along $B\times \{0\}$. The projection $X\times \C\to\C$
induces a flat projection $\sigma\colon\mathcal X\to \C$, and one has
$\sigma^{-1}(t)=X$ if $t\ne 0$, and
$\sigma^{-1}(0)=X\cup E_{B}$. Furthermore, in this latter case
$X\cap E_{B}$ is the curve $B$ in $X$, and the curve $B_\infty$
in $E_{B}$. Note that if both $X$ and $B$ are real, and if we equip
$\C$ with the standard complex conjugation, then the map
$\sigma$ is a real map.

Let $C_0=C_X\cup C_B$ be an algebraic curve in $X\cup E_{B}$ such
that:
\begin{enumerate}
\item $C_X\subset X$ is  nodal and intersects $B$ transversely;
\item $C_B\subset E_B$ is  nodal and intersects  $B_\infty$ transversely; 
let $a = [C_B] \circ [F]$
in $H_2(E_B;\Z)$; 
\item $C_X\cap B= C_B\cap B_\infty = C_X \cap C_B$.
\end{enumerate}
In the following two propositions, we use \cite[Theorem 2.8]{ShuTyo06} to ensure
the existence of a deformation $C_t$  in $\sigma^{-1}(t)$ within the
linear system $|C_X+ aB|$ of the curve
$C_0$ in some particular instances.
We denote by  $\mathcal P$  the set of nodes of $C_0\sminus (X\cap E_B)$,
and by $\mathcal I_X$ (resp. $\mathcal I_B$)  the sheaf of ideals of
$\mathcal P\cap X$ (resp. $\mathcal P\cap E_B$).

\begin{prop}\label{prop:dnc1}
In the notation above, suppose that   $X\subset \CP^3$ is a quadric ellipsoid, and that $B$ is a 
real
hyperplane section.
If  $C_0$ is a finite real algebraic curve,
then there exists a  finite real algebraic
curve $C_1$ in $X$ in the linear system $|C_X+aB|$  such that
\[
|\R C_1|=|\R C_0|.
\]
\end{prop}
\begin{proof}
One has the following short exact sequence of sheaves
\[
0\longrightarrow \mathcal O(C_X)\otimes \mathcal I_X
\longrightarrow  \mathcal O (C_X)
\longrightarrow \mathcal O_{\mathcal P\cap  X}\longrightarrow 0.
\]
(To shorten the notation, we abbreviate ${\mathcal O}(D) = {\mathcal O}_X(D)$
for a divisor $D \subset X$ when the ambient variety $X$ is understood.) 
Since  $H^1(X,\mathcal O(C_X))=0$, one obtains
the following exact sequence
\[
0\longrightarrow H^0(X,\mathcal O(C_X)\otimes \mathcal I_X)
\longrightarrow  H^0(X,\mathcal O(C_X))
\longrightarrow
H^0(\mathcal P\cap X,\mathcal O_{\mathcal P\cap  X})
\longrightarrow H^1(X,\mathcal O(C_X)\otimes \mathcal I_X)
\longrightarrow 0.
\]
The surface $\CP^1\times\CP^1$ is toric
and it is a classical application of Riemann-Roch Theorem that
$ H^0(X,\mathcal O(C_X)\otimes \mathcal I_X) $ has codimension
$|\mathcal P\cap  X|$ in $H^0(X,\mathcal O(C_X))$ (see for example
\cite[Lemma 8 and Corollary 2]{Sh2}). Since
$h^0(\mathcal P\cap X,\mathcal O_{\mathcal P\cap  X})=|\mathcal P\cap
X|$, we deduce that
\[
H^1(X,\mathcal O(C_X)\otimes \mathcal I_X)=0.
\]
The curve $B$ is rational, and
the surface $E_B$ is the surface $\Sigma_2$. In particular,
$E_B$ is a toric surface and  $B_\infty$
is an irreducible component of its toric boundary. Hence we analogously obtain
\[
H^1(E_B,\mathcal O(C_B-B_\infty)\otimes \mathcal I_B)=0.
\]
Hence by \cite[Theorem 3.1]{ShuTyo06}, the proposition is now a
consequence of \cite[Theorem 2.8]{ShuTyo06}.
\end{proof}


Recall that $H^0(E_B,\mathcal O(C_B)\otimes \mathcal I_B)$ is the
set of elements of  $H^0(E_B,\mathcal O(C_B))$ vanishing on $\mathcal
P\cap E_B$.

\begin{prop}\label{prop:dnc2}
Suppose that   $X=\CP^2$, that $B$ is a non-singular real cubic
curve, and that $C_X=\varnothing$.
If  $C_B$ is a finite real algebraic curve and if
$H^0(E_B,\mathcal O(C_B)\otimes \mathcal I_B)$ is of codimension
$|\mathcal P|$ in
$H^0(E_B,\mathcal O(C_B))$, then there exists a  finite real algebraic
curve $C_1$ in $\CP^2$ of degree $3a$ such that
\[
|\R C_1|=|\R C_B|.
\]
\end{prop}
\begin{proof}
  Recall that $E_B$ is a ruled surface over $B$, {\it i.e.}, is  equipped
  with a $\CP^1$-bundle  $\pi \colon E_B\to B$.
  By \cite[Lemma 2.4]{Har77}, we have
  \[
  H^i(E_B,\mathcal O(C_B))\simeq H^i(B_\infty, \pi_*\mathcal O(C_B)), 
  \qquad 
  i\in\{0,1,2\}.
  \]
  In particular the short exact sequence of sheaves
\[
0\longrightarrow \mathcal O(C_B-B_\infty)
\longrightarrow  \mathcal O(C_B)
\longrightarrow \mathcal O_{B_\infty}\longrightarrow 0
\]
gives rise to
the exact sequence
\[
0\longrightarrow H^0(E_B,\mathcal O(C_B-B_\infty))
\longrightarrow  H^0(E_B,\mathcal O(C_B))
\longrightarrow
H^0(B_\infty,\mathcal O_{B_\infty})
\longrightarrow
\]
\[
\longrightarrow H^1(E_B,\mathcal O(C_B-B_\infty))
\longrightarrow  H^1(E_B,\mathcal O(C_B))\xrightarrow[]{\mbox{ } \iota_1
  \mbox{ }}
H^1(B_\infty,\mathcal O_{B_\infty})\longrightarrow 0.
\]
Furthermore, by \cite[Proposition 3.1]{GalPur96} we have
$H^1(E_B,\mathcal O(C_B-B_\infty))=0$, hence
the map $\iota_1$ is an isomorphism.

On the other hand,   the short exact sequence of sheaves
\[
0\longrightarrow \mathcal O(C_B)\otimes \mathcal I_B
\longrightarrow  \mathcal O(C_B)
\longrightarrow \mathcal O_{\mathcal P}\longrightarrow 0
\]
gives rise to
the exact sequence
\[
0\longrightarrow H^0(E_B,\mathcal O(C_B)\otimes \mathcal I_B)
\longrightarrow  H^0(E_B,\mathcal O(C_B))
\xrightarrow[]{\mbox{ } r_1\mbox{ } }
H^0(\mathcal P,\mathcal O_{\mathcal P})
\longrightarrow
\]
\[
\longrightarrow
H^1(E_B,\mathcal O(C_B)\otimes \mathcal I_B)
\xrightarrow[]{\mbox{ }\iota_2\mbox{ }} H^1(E_B,\mathcal O(C_B))\longrightarrow 0.
\]
By assumption, the map $r_1$ is surjective, so we deduce that
the map $\iota_2$ is an isomorphism.

We denote by $\widetilde{\mathcal L}_0$ 
 the invertible sheaf on the disjoint union of $E_B$ and $\CP^2$ and
 restricting to $\mathcal O(C_B)$ and $\mathcal O_{\CP^2}$ on $E_B$ and
 $\CP^2$ respectively.
 Finally, we denote by $ \mathcal L_0$ the invertible sheaf on
 $\sigma^{-1}(0)$ for which $C_0$ is the zero set of a section.
 The natural short exact sequence
 \[
 0\longrightarrow \mathcal L_0\otimes \mathcal I_B \longrightarrow
 \widetilde{\mathcal L}_0\otimes \mathcal I_B \longrightarrow \mathcal O_B\longrightarrow 0 
 \]
 gives rise to the long exact sequence
 \[
 0\longrightarrow H^0(\sigma^{-1}(0),\mathcal L_0 \otimes \mathcal I_B)\longrightarrow
 H^0(E_B,\mathcal O(C_B)\otimes \mathcal I_B)\oplus H^0(\CP^2,\mathcal O_{\CP^2})
 \xrightarrow[]{\mbox{ }r_2 \mbox{ }} H^0(B,\mathcal O_{B})
 \longrightarrow
 \]
 \[
 \longrightarrow H^1(\sigma^{-1}(0),\mathcal L_0\otimes \mathcal I_B
 )\longrightarrow H^1(E_B,\mathcal O(C_B)\otimes \mathcal I_B)
 \xrightarrow[]{\mbox{ }\iota \mbox{ }} H^1(B,\mathcal O_{B})
  \longrightarrow H^2(\sigma^{-1}(0),\mathcal L_0 \otimes \mathcal I_B)\longrightarrow 0.
  \]
  The restriction of the map $r_2$ to the second factor
  $H^0(\CP^2,\mathcal O_{\CP^2}) $ is clearly an isomorphism, hence we
  obtain the exact sequence
  \[
  0\longrightarrow H^1(\sigma^{-1}(0),\mathcal L_0\otimes \mathcal I_B
 )\longrightarrow H^1(E_B,\mathcal O(C_B)\otimes \mathcal I_B)
 \xrightarrow[]{\mbox{ }\iota \mbox{ }} H^1(B,\mathcal O_{B})
  \longrightarrow H^2(\sigma^{-1}(0),\mathcal L_0 \otimes \mathcal I_B)\longrightarrow 0.
  \]
Since $\iota=\iota_1\circ \iota_2$ is an isomorphism, we deduce that
  $ H^1(\sigma^{-1}(0),\mathcal L_0\otimes \mathcal I_B)=0$. Now the
proposition follows from \cite[Theorem 2.8]{ShuTyo06}.
\end{proof}

\section{Finite curves in $\CP^2$}\label{sec:plane}


In the case $X = \CP^2$,
Theorem \ref{thm:upper con} and Corollary \ref{cor:pet}
specialize as follows.
\begin{thm}\label{thm:cp2}
 Let
 $C \subset \CP^2$ be  a
finite real algebraic curve of degree $2k$.
Then,
\begin{gather}\label{eq:cp2 genus}
|\R C|\le k^2+g(C)+1, \\
\label{eq:cp2}
|\R C|\le \frac{3}{2}k(k-1) + 1.
\end{gather}
\end{thm}


In the rest of this section, we discuss the sharpness of these bounds.

\subsection{Asymptotic constructions}\label{asymptotic-construction}
%
%
The following asymptotic lower bound holds for any projective toric surface
with the standard real structure.

\begin{thm}\label{thm:toric}
Let $\Delta \subset \R^2$ be a convex lattice polygon, and let $X_\Delta$
be the associated toric surface. Then, there exists
a sequence of finite real algebraic curves $C_k \subset X_\Delta$ with the Newton polygon
$\Delta(C_k) = 2k\Delta$, such that
$$\lim_{k\to\infty}\frac{1}{k^2}|\R C_k|= \frac{4}{3}\Area(\Delta),$$
where $\Area(\Delta)$ is the lattice area of $\Delta$.
\end{thm}

\begin{rem}\label{very-new-remark}
In the settings of Theorem \ref{thm:toric}, assuming $X_\Delta$ smooth, the asymptotic upper bound
for finite real algebraic curves $C \subset X_\Delta$ with $\Delta(C) = 2k\Delta$
is given by Theorem \ref{thm:upper con}:
$$
|\R C| \lesssim \frac{3}{2}\Area(\Delta).
$$
\end{rem}

\begin{proof}[Proof of Theorem \ref{thm:toric}]
There exists a (unique) real rational
cubic $C \subset (\C^*)^2$ such that
\begin{itemize}
\item $\Delta(C)$ is the triangle with the vertices $(0,0)$, $(2,1)$, and $(1,2)$;
\item the coefficient of the defining polynomial $f$ of $C$
at each corner of $\Delta(C)$ equals $1$;
\item $\R C \cap \R_{> 0}^2$ is a single solitary node.
\end{itemize}

\begin{figure}[h!]
\begin{center}
\includegraphics[width=5cm, angle=0]{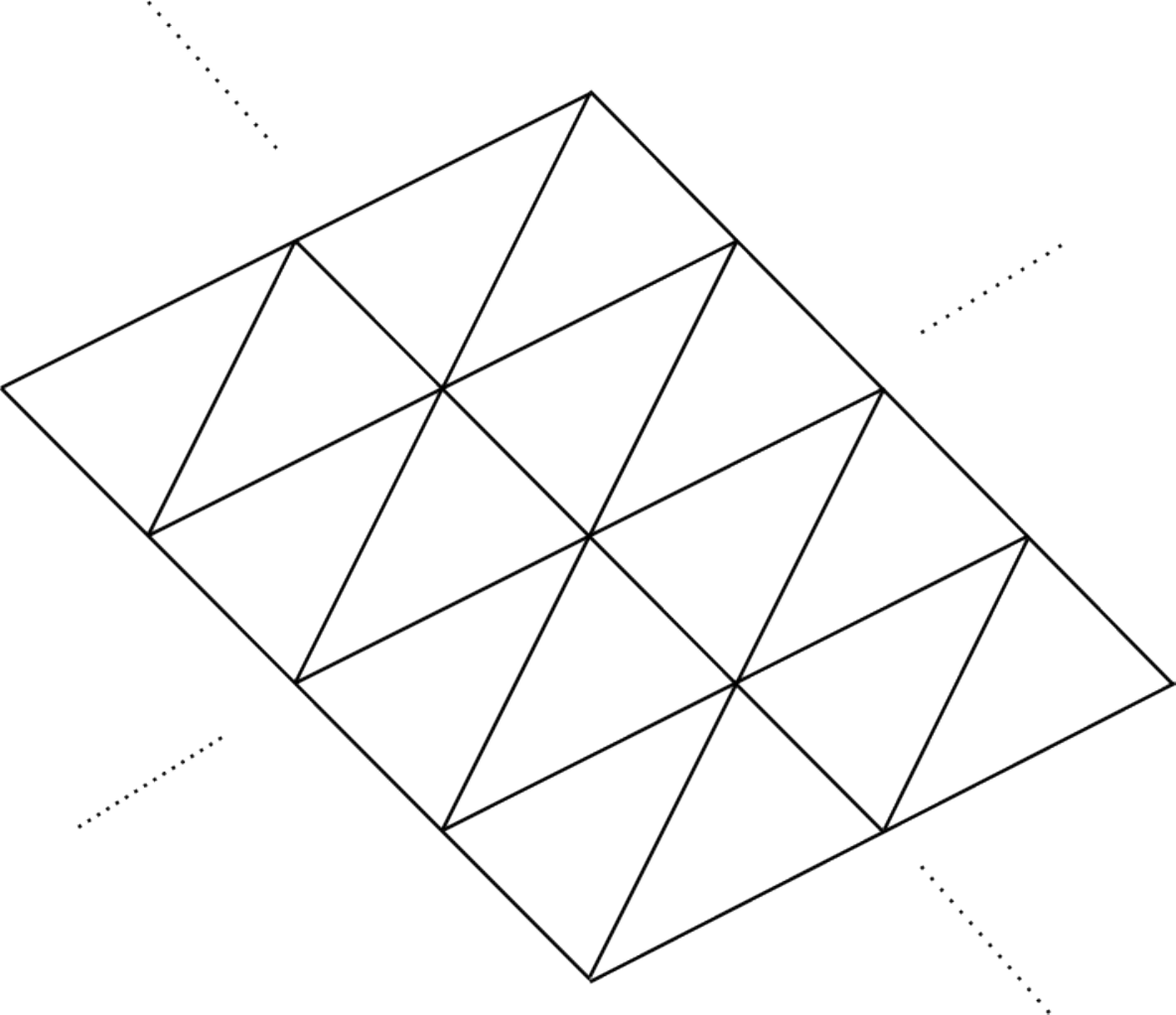}
\end{center}
\caption{}
\label{fig:asy gluing}
\end{figure}

Figure \ref{fig:asy gluing} shows a tilling of $\R^2$ by lattice congruent copies of $\Delta(C)$.
Intersecting this tilling with $k\Delta$ and making an appropriate adjustment
in the vicinity of the boundary, we obtain a convex subdivision of $k\Delta$
containing $\frac{1}{3}k^2 \Area(\Delta) + O(k)$ copies of $\Delta(C)$.
Now, for each of these copies, we consider an appropriate monomial multiple
of either $f(x, y)$ or $f(1/x, 1/y)$.
Applying Theorem \ref{patchworking}, we obtain a real polynomial $f_k$
whose zero locus in
$\R_{> 0}^2$
consists of $\frac{1}{3}k^2 \Area(\Delta) + O(k)$ solitary nodes.
There remains to let $C_k = \{f_k(x^2, y^2) = 0\}$.
\end{proof}

\begin{cor}\label{thm:asy}
There exists a sequence of finite real algebraic curves $C_k \subset \CP^2$,
$\deg C_k = 2k$, such that
$$\lim_{k \to +\infty}\frac{1}{k^2}|\R C_k|= \frac{4}{3}.$$
\end{cor}


In the next theorem, we tweak the ``adjustment in the vicinity of the boundary''
in the proof of Theorem \ref{thm:toric} in the case $X_\Delta = \CP^2$.

\begin{thm}\label{thm:exact}
For any integer $k\ge 3$, there exists a finite real algebraic curve $C \subset \CP^2$ of degree $2k$
such that
$$|\R C| =
\begin{cases}
12l^2 -4l +2 & \mbox{if }k=3l, \\
12l^2 +4l +3 & \mbox{if }k=3l +1, \\
12l^2 +12l +6& \mbox{if }k=3l +2.
 \end{cases}
 $$
\end{thm}
\begin{proof}
Following the proof of Theorem \ref{thm:toric},
we use the subdivision of the triangle $k\Delta$ (with the vertices $(0, 0)$, $(k, 0)$, and $(0, k)$)
shown in Figure \ref{fig:exact}.
In the $t$-axis ($t = x$ or $y$),
each segment of length $1$, $2$ or $3$
bears an appropriate monomial multiple
of $1$, $(t - 1)^2$ or $(t - 1)^2(t + 1)$, respectively.
Thus, each segment $\ell$ of length $2$ or $3$ gives rise
to a point of tangency of the $t$-axis and the curve $\{f_k = 0\}$,
resulting in two extra solitary nodes of $C_k$.
Similarly, each vertex of $k\Delta$ contained in a segment of length $1$
gives rise to an extra solitary node of $C_k$.
\begin{figure}[h!]
\begin{center}
\begin{tabular}{ccc}
\includegraphics[width=4cm, angle=0]{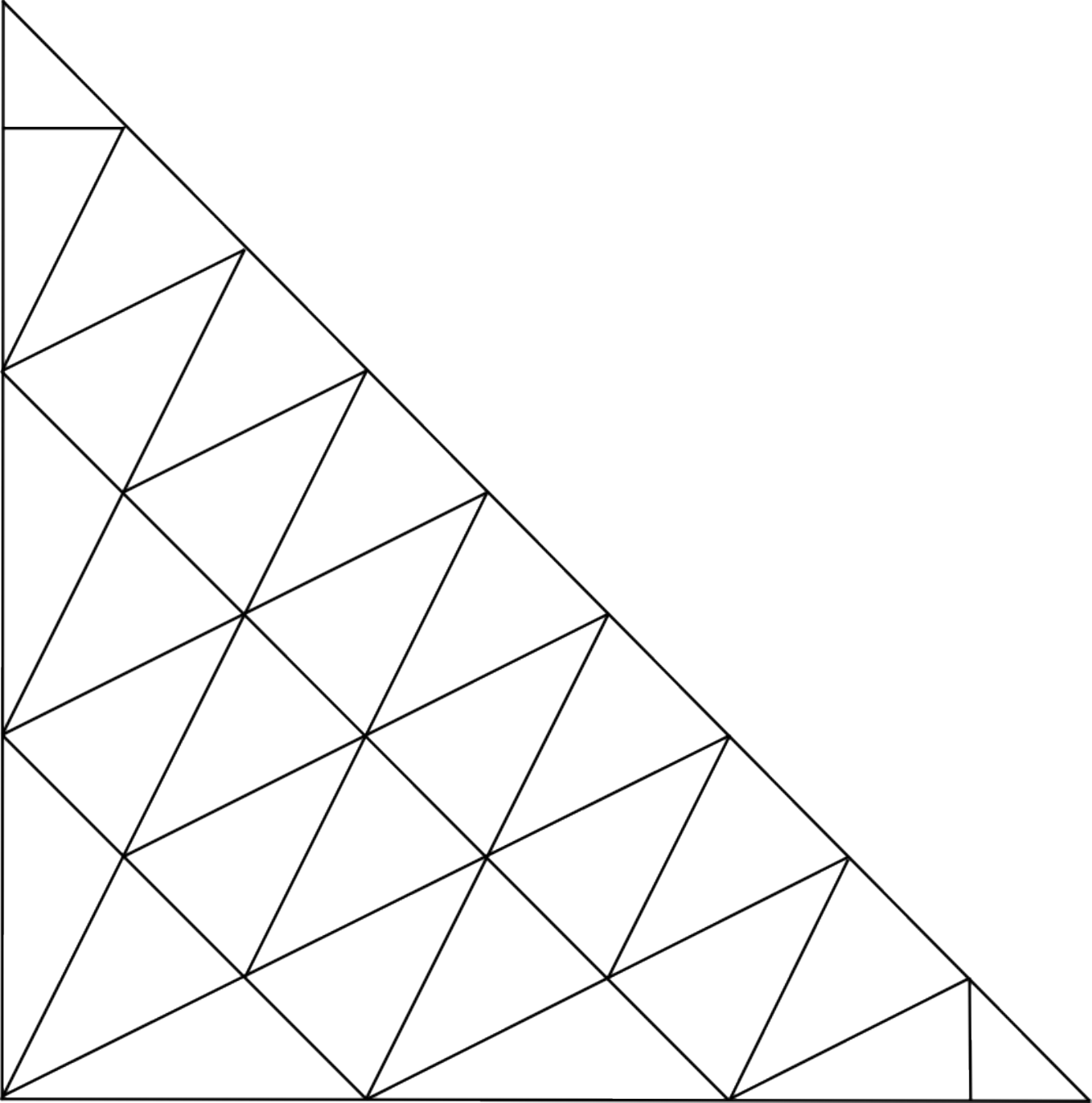}
\hspace{3ex} &
\includegraphics[width=4cm, angle=0]{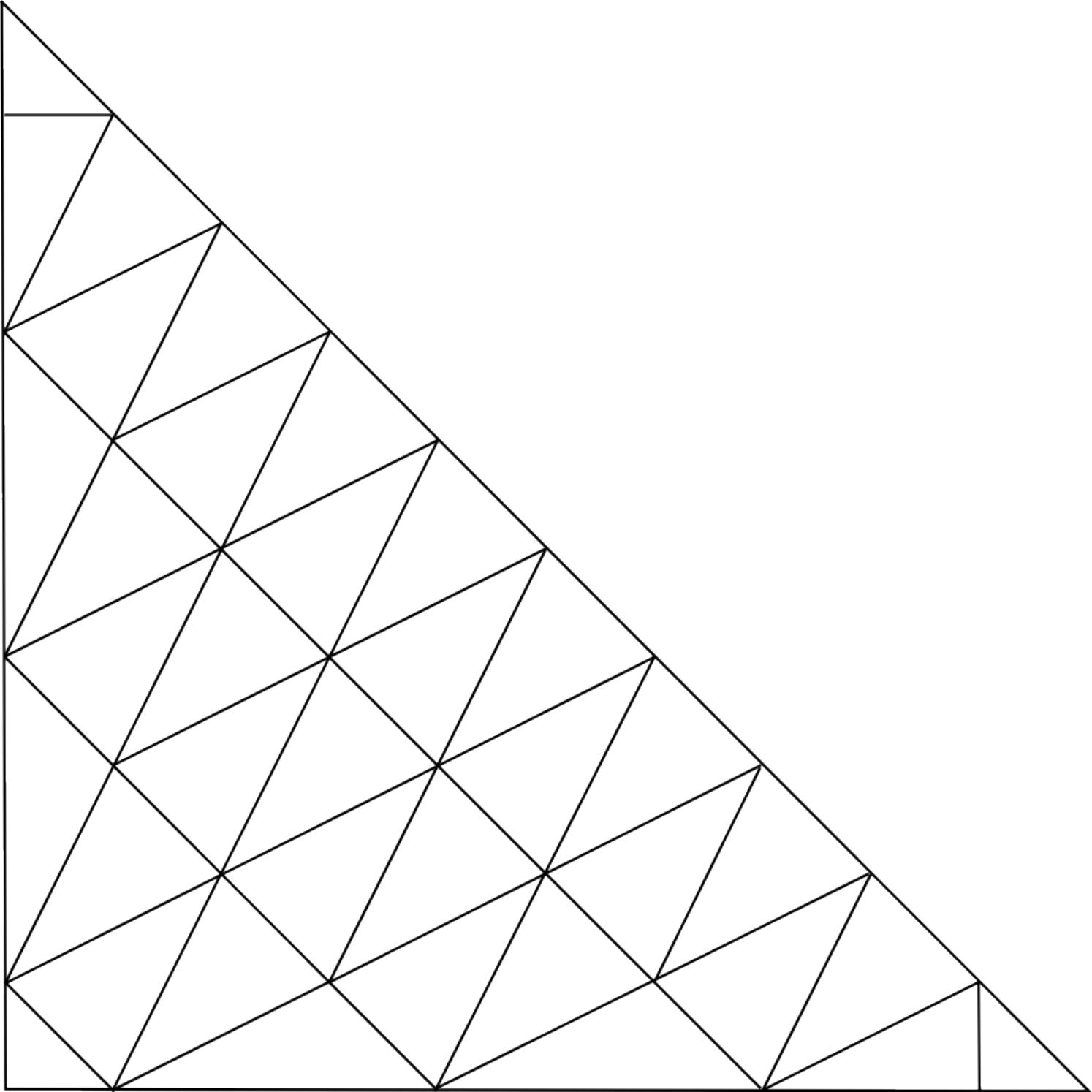}
\hspace{3ex} &
\includegraphics[width=4cm, angle=0]{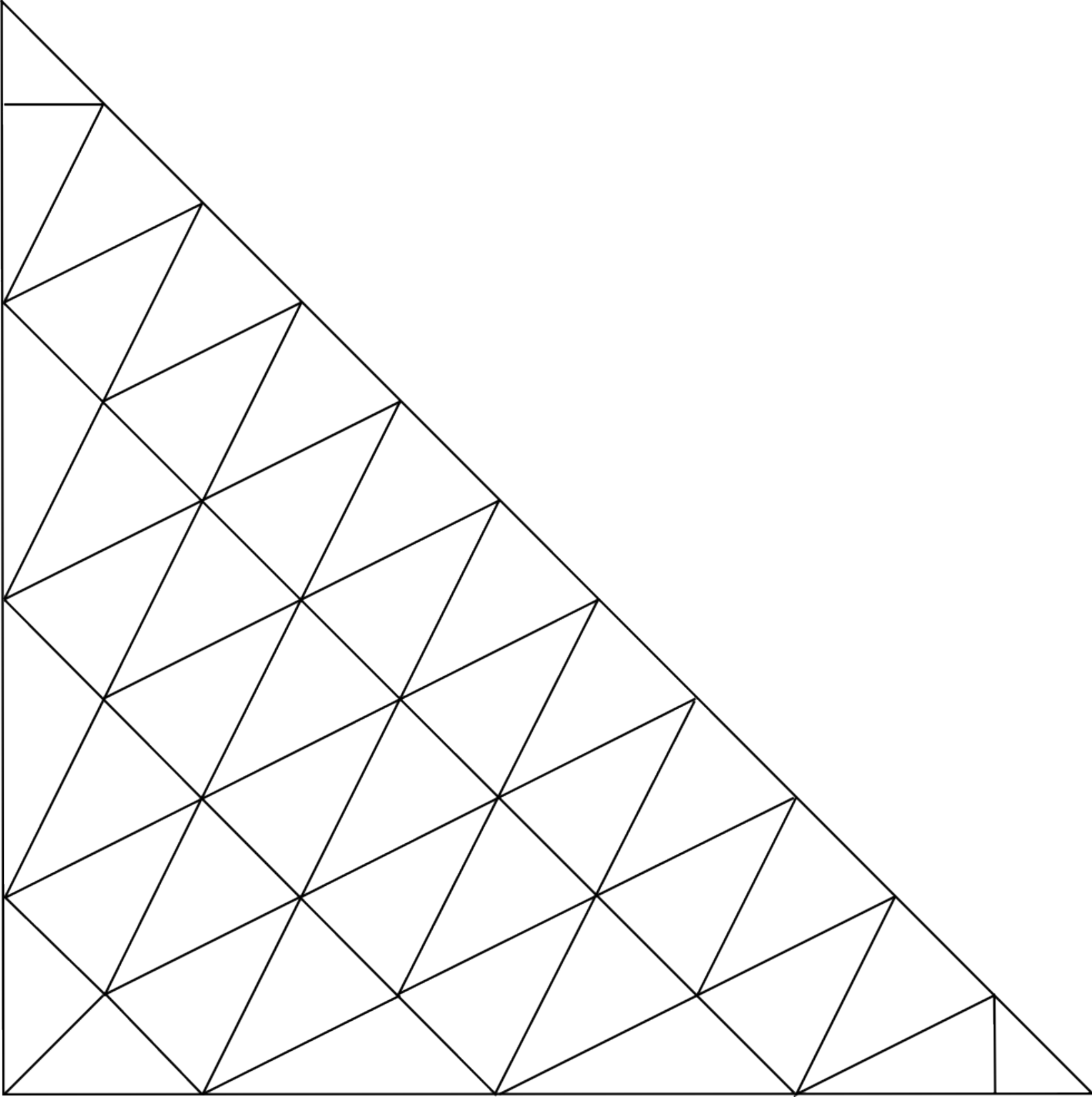}
\\ a) $k=3l$
&b) $k=3l+1$ & c) $k=3l+2$
\end{tabular}
\end{center}
\caption{}
\label{fig:exact}
\end{figure}
%
%
\end{proof}

%

\begin{rem}\label{remark-non-patch}
The
construction of Theorem \ref{thm:exact} for $k= 3,4$ can easily be
performed without using the patchworking
technique.
\end{rem}

\subsection{A curve of degree $12$}\label{degree12}
The construction given by Theorem \ref{thm:exact}
is the best known if $k \leq 5$. If $k = 6$, we can improve it by $2$ more units.

\begin{prop}\label{prop:better 12}
There exists a finite real algebraic curve $C \subset \CP^2$ of degree $12$
such that
$$|\R C|=45.$$
\end{prop}

\begin{proof}
  Let $C'=C'_9(1)$ be a finite real algebraic curve in
  $\Sigma_{9}(\O')$ as in
  Lemma~\ref{prop:positive genus}. Let us denote by $\P$ the set of
  nodes of $C'$, and by $\mathcal I$ the sheave of ideals on
  $\Sigma_{9}(\O')$
  defining $\P$.
Since $\R B \ne \varnothing$,
there exists a real
line
bundle $\L_0$ of degree $3$ over~$B$ such that
$\O'=\L_0^{\otimes 3}$. This bundle $\L_0$
embeds~$B$ into $\CP^2$
as a real cubic curve for which $\L_0^{\otimes3}=\O$ is the normal bundle.
The proposition will then follow from Proposition \ref{prop:dnc2} once
we prove that
$H^0(\Sigma_{9}(\O'),\mathcal O(4B_0)\otimes \mathcal I)$ is of codimension
$45$ in
$H^0(\Sigma_{9}(\O'),\mathcal O(4B_0))$. Let us show that this is
indeed the case, {\it i.e.}, let us show that given any node $p$ of $C'$, there
exists an algebraic curve in $\O(4B_0)$ on  $\Sigma_{9}(\O')$
passing through all nodes of
$C'$ but $p$.
Recall that there exists a real double covering
$\rho\colon\Sigma_{9}(\O')\to \Sigma_{9}(\O'^{\otimes2})$ ramified along
$B_0\cup B_\infty$ with respect to which $C'$ is symmetric, and that
$C'$ has $18$ pairs of symmetric nodes and $9$ nodes on $B_0$.

By Riemann-Roch Theorem, for any line bundle $\O$ over $B_0$ of degree
$n\ge 1$, and given any set $\P$ of $n-2$ points on distinct fibers of
$\Sigma_{n}(\O)$ and any disjoint finite subset $\overline \P$ of $\Sigma_{n}(\O)$,
there exists an algebraic  curve in $\O(B_0)$ containing $\P$ and
avoiding $\overline \P$.
As a consequence, there exists a symmetric curve in
$\O(2B_0)$  on  $\Sigma_{9}(\O')$  passing through any $16$ pairs of
symmetric nodes of $C'$ and avoiding all other nodes of $C'$.
Altogether, we see that given any node $p$ of $C$',
there exists a reducible curve in $\O(4B_0)$ on
$\Sigma_{9}(\O')$,
consisting in the union of a symmetric curve in $\O(2B_0)$ and two
curves in $\O(B_0)$, and
passing through all nodes of $C'$ but $p$.
\end{proof}

\subsection{Curves of low genus}
Here we show that inequality $(\ref{eq:cp2 genus})$ of Theorem \ref{thm:cp2} is sharp when the
degree is large compared to the genus.

\begin{thm}\label{thm:cp2 low}
Given integers $k\ge 3$ and $0 \leq g \le k-3$, there exists a
finite real algebraic curve $C \subset \CP^2$ of degree $2k$
and genus $g$ such that
$$|\R C| = k^2+g+1.$$
\end{thm}


\begin{proof}
Consider a real rational curve $C_1 \subset C^2$ with the following properties:
\begin{itemize}
\item the Newton polygon of $C_1$ is the triangle with the vertices $(0,0)$, $(0, k - 2)$ and
$(2k-4, 0)$,
\item $C_1$ intersects the axis $y=0$ in a single point with
multiplicity $2k-4$,
\item $\R C_1 \cap \{y>0\}$
consists of $\frac{1}{2}(k-2)(k-3)$ solitary nodes.
\end{itemize}
Such a curve exists: for example, one can take
a rational simple Harnack curve with the prescribed Newton polygon
(see \cite{Mik11,KenOko06,Bru14b}).
Shift the Newton polygon $\Delta(C_1)$ by $2$ units up and place in the trapezoid
with the vertices $(0, 0)$, $(2k, 0)$, $(2k - 4, 2)$, $(0, 2)$ a defining polynomial
of the curve $\widetilde C_{1,k-2,g+1}$ given by Lemma \ref{lem:deg2}.
Applying Theorem \ref{patchworking}, we obtain a
real rational curve $C_2 \subset \C^2$ such that
\begin{itemize}
\item $\R C_2\cap \{y>0\}$ consists of $\frac{1}{2}(k-2)(k-3) + 2k+g-2$
solitary nodes,
\item $C_2$ intersects the line $y=0$ in $k-g-1$ real points
of multiplicity 2, and in $g+1$ additional pairs of complex
conjugated points.
\end{itemize}

If $C_2$ is given by an equation $f(x,y)=0$ positive on
$y>0$, we define $C$ as the curve
$f(x,y^2)=0$.
Each node $p \in \{y > 0\}$ of
$C_2$
gives
rise to two solitary real nodes of $C$,
and each tangency point of
$C_2$ and the axis $y = 0$
gives rise to an extra solitary node of $C$.
The genus $g(C) = g$ is given by the Riemann--Hurwitz formula
applied to the double covering $C \to C_2$:
its normalization is branched at
the $2(g + 1)$ points of transverse intersection of $C_2$ and the axis $y = 0$.
 \end{proof}

%
%

\section{Finite curves in real ruled surfaces}\label{sec:hirz}

We use the notation $B,\O,B_0, F, \Sigma _n(\O)$ introduced in Section
\ref{dessins}.
A real algebraic curve $C$ in $\Sigma_n(\O)$
realizing the class $u[B_0]+ v[F]\in H_2(\Sigma_n(\O);\Z)$
may be finite only if both $u = 2a$ and $v = 2b$ are even. 
General results of the previous sections
specialize as follows.

\begin{thm}
  \label{thm:hirz}
Let
$C \subset \Sigma_n(\O)$ be  a
finite real algebraic curve,
$[C] = 2a[B_0]+ 2b[F]\in H_2(\Sigma_n(\O);\Z)$, $a > 0$, $b > 0$.
Then,
\begin{gather}\label{eq:hirz genus}
|\R C|\le na^2+2ab + g(C)+1 -2g(B), \\
\label{eq:hirz}
|\R C|\le \frac{1}{2}{na(3a - 1)} + 3ab - (a + b) + 1 +(a-1)g(B).
\end{gather}
\end{thm}
\begin{proof}
The statement is an immediate consequence of Theorem
  \ref{thm:upper} and Corollary \ref{cor:pet}:
due to Lemma \ref{lem:real double}, we can
choose
$\R X_+=\R X$ and $\R X_-=\varnothing$.
\end{proof} 



As in the case of $\CP^2$, we
do not know
whether the upper bounds $(\ref{eq:hirz genus})$ and
$(\ref{eq:hirz})$ are sharp
in general.
In the rest of the section,
we discuss the special cases
of small $a$ or small genus.
The two next propositions easily generalize to ruled surfaces over a base
of any genus (in the same sense as explained after Lemma \ref{arbitrary_genus_curve}).
For simplicity, we confine ourselves to the case of a 
rational
base.

\begin{prop}[$a = 1$]\label{prop:a=1}
Given integers $b, n \ge 0$, there exists
a finite real algebraic curve $C \subset \Sigma_n$ of bidegree $(2, 2b)$ such that
$|\R C| = n + 2b$.
\end{prop}

\begin{proof}
A collection of $n + 2b$ generic real points in $\Sigma_n$
determines a real pencil of curves of bidegree $(1, b)$,
and one can take for $C$
the union of two complex conjugate members of this pencil.
\end{proof}

\begin{prop}[$a = 2$]\label{prop:deg4}
Given integers $b, n \ge 0$, and  $-1 \le g \le n + b - 2$,
there exists a finite real
algebraic curve $C \subset \Sigma_n$ of bidegree $(4, 2b)$
and
genus $g$ such that
$$|\R C|=4n + 4b + g + 1.$$
In particular, if $b + n \geq 1$, then there exists a finite real
algebraic curve $C \subset \Sigma_n$ of bidegree $(4, 2b)$  such that
$$|\R C| = 5n + 5b - 1.$$
\end{prop}

\begin{proof}
We
argue as in the proof of Lemma~\ref{prop:positive genus}, starting from
the curve $\widetilde C_{n,b,g+1}$ given by Lemma~\ref{lem:deg2}.
The genus $g(C)$ is
computed
by the Riemann--Hurwitz formula.
\end{proof}





All rational ruled surfaces
are
toric, and
Theorem \ref{thm:toric}
takes the following form.

\begin{thm}
Given integers $a > 0$ and $b \ge 0$, there exists
a sequence of finite real algebraic curves $C_k \subset \Sigma_n$
of bidegree $(ka, kb)$ such that
$$\lim_{k \to +\infty}\frac{1}{k^2}|\R C_k|= \frac{4}{3}(na^2+2ab).\eqno\qed$$
\end{thm}


Furthermore, the proof of Theorem \ref{thm:cp2 low}
extends literally to
curves in $\Sigma_n$.

\begin{thm}[low genus]\label{thm:low hirz}
Given integers $a > 0$, $b,n \ge 0$, and $-1 \leq g \le n(a - 1) + b - 2$,
there exists a
finite real algebraic curve $C \subset \Sigma_n$ of bidegree $(2a, 2b)$ and genus $g$
such that
$$|\R C|= na^2 + 2ab + g + 1. \eqno\qed$$
\end{thm}

\section{Finite curves in the ellipsoid}\label{ellipsoid}

The algebraic surface $\Sigma_0=\CP^1\times\CP^1$ has two real
structures with non-empty real part, namely
$c_h(z,w)=(\bar z, \bar w)$ and
$c_e(z,w)=(\bar w, \bar z)$. The first one was 
considered in Section \ref{sec:hirz}. In this section, $\Sigma_0$ is
assumed 
equipped with the real structure $c_e$, and we
have $\R \Sigma_0=S^2$. 

\subsection{General bounds}\label{sec:general_bounds} 
Let $e_1$ and $e_2$ be the classes in $H_2(\Sigma_0;\Z)$
represented by the two rulings. The action of $c_e$ on $H_2(\Sigma_0;\Z)$ is
given by $c_e(e_i)=-e_{3-i}$, and so $\sigma^-_{\text{\rm inv}}(\Sigma_0, c_e)=1$.

The classes in $H_2(\Sigma_0;\Z)$ realized by real algebraic curves are those of the form
$m(e_1+e_2)$. 
For any $m\ge 1$, a real algebraic curve of bidegree
$(m,m)$ may have finite real part. 

\begin{thm}\label{thm:quad}
 Let
 $C$ be  a
reduced  finite real algebraic curve in $(\Sigma_0,c_e)$ of bidegree $(m,m)$,
with $m\ge 2$. 
 Then
 \begin{equation}\label{eq:quad genus}
   |\R C|\le \left\{\begin{array}{ll} 2k^2+  g(C)+3 &\mbox{if }
   m=2k\\ \\
    2k^2+4k + g(C) &\mbox{if } m=2k+1\end{array}\right. .
\end{equation}
 In particular we have
 \begin{equation}\label{eq:quad}
   |\R C|\le\left\{\begin{array}{ll} 3k^2-2k+2 &\mbox{if }
   m=2k\\ \\
    3k^2+2k &\mbox{if } m=2k+1\end{array}\right. .
\end{equation}
\end{thm}
\begin{proof}
In order to apply Theorem \ref{thm:upper con},
we note that $T_{2,1}(\Sigma_0)=-h^{1,1}(\Sigma_0)=-2$ and that the real locus of $(\Sigma_0, c_e)$ being a sphere, $\chi(\R \Sigma_0)=2$.

The case when $m=2k$ is then provided by Theorem \ref{thm:upper con} and
Corollary \ref{cor:pet}. Indeed, in this case, $[C]=m(e_1+e_2)=2k(e_1+e_2)$ and letting $e=k(e_1+e_2)$, we get $e^2=2k^2$ and $e\cdot c_1(\Sigma_0)=2k(e_1+e_2)(e_1+e_2)=4k$.

So suppose that $m=2k+1$ and let $p\in \R C$.
Let $E_1$ and $E_2$ be a pair of conjugate generatrices which meet $C$ at $p$.
Let $\widetilde C=C\cup E_1\cup E_2$ and let $\overline{C}$ be the strict transform of $\widetilde C$ in the blow-up 
$\overline{\Sigma}_0$ of $\Sigma_0$ at $p$. The class of the auxiliary curve $\widetilde C$ in $H_2(\Sigma_0;\Z)$ is then $[\widetilde C]=2(k+1)(e_1+e_2)$. Let $e=(k+1)(e_1+e_2)$, we get $e^2=2(k+1)^2$. Let $\overline{e}$ be half the class of $\overline{C}$ in $H_2(\overline{\Sigma}_0;\Z)$, we get $\overline{e}^2\le e^2-4$, as the point $p$ is of multiplicity at least $4$ in $\widetilde C$.
Furthermore, we have $g(\overline{C})=g(\widetilde C)=g(C)-2$ and
$|\R C| = |\R \widetilde C|=|\R \overline{C}|+1$.
In order to apply Theorem \ref{thm:upper con} for the curve $\overline{C}$ on $\overline{\Sigma}_0$,
it remains to note that $T_{2,1}(\overline{\Sigma}_0)=T_{2,1}(\Sigma_0)-1$ and $\chi(\R \overline{\Sigma_0})
=\chi(\R \Sigma_0)-1$.
Hence we obtain
$(\ref{eq:quad genus})$ from Theorem \ref{thm:upper con} applied to the curve $\overline{C}$ on $\overline{\Sigma}_0$. 

To get $(\ref{eq:quad})$, it suffices to remark that, $C$ being a curve of bidegree $(2k+1,2k+1)$, we have $g(C)\le 4k^2-|\R C|$.
\end{proof}

\begin{rem}
Let us consider the following problem:
given a smooth real projective surface  $(X,c)$ and
a homology class $d\in H_2(X;\Z)$,  what is the maximal possible number
of intersection points between $C$ and $\R X$ for a non-real algebraic
curve $C$ in $X$ realizing the class $d$?

Since any two distinct irreducible algebraic curves in $X$ intersect
positively, any non-real irreducible algebraic curve $C$ in $X$
intersects $c(C)$ in $-[C]\cdot c_*[C]$ points, and so intersects
$\R X$ in at most $-[C]\cdot c_*[C]$ points.
It is easy to see  that this upper bound is sharp in $\CP^2$.
Interestingly, Theorem \ref{thm:quad}
shows that this trivial upper bound is not sharp in the case of  the quadric ellipsoid.

Any irreducible algebraic curve $C$ in $\Sigma_0$ realizing the class $(m-1,1)$
with $m\ge 3$ is non real and rational. Since the union of $C$ and
$c_e(C)$ is a real algebraic curve of geometric genus $-1$ realizing the
class $(m,m)$, Theorem
\ref{thm:quad}
implies that
$$|C\cap \R \Sigma_0|\le \left\{\begin{array}{ll} 2k^2+  2 &\mbox{if }
m=2k\\ \\
2k^2+4k - 1 &\mbox{if } m=2k+1\end{array}\right.,  $$
whereas $(m-1,1)\cdot (1,m-1)=m^2 -2m+2$
is at least
twice as large.
\end{rem}

Next theorem is an immediate consequence of Theorem \ref{thm:toric} and Proposition \ref{prop:dnc1} 
\begin{thm}
There exists
a sequence of finite real algebraic curves $C_m$
   of bidegree $(m,m)$ in the quadric ellipsoid 
   such that
$$\lim_{m\to\infty}\frac{1}{2m^2}|\R C_m|= \frac{4}{3}.$$ 
\end{thm}

\subsection{Curves of low bidegree}


Next statement shows in particular that Theorem \ref{thm:quad} is not
sharp for $m=2$ and $m=5$.

\begin{prop}\label{prop:quadric}
  For $m\le 5$,  the maximal possible value $\delta_e(m)$
  of $|\R C|$ for a finite real algebraic curve of
bidegree $(m,m)$ in the quadric ellipsoid 
is 
$$\begin{array}{c | c| c| c| c| c}
  m & 1& 2& 3& 4& 5
  \\\hline \delta_e(m) & 1&  2& 5& 10& 15 
\end{array}$$


\end{prop}
\begin{proof}
We start by constructing  real algebraic
 curves with a number of real points as stated in the
proposition.
 For $m\le 4$, such a curve is constructed by taking the union of two complex
conjugated curves of bidegree $(m-1,1)$ and $(1,m-1)$ intersecting $\R\Sigma_0$
in $(m-1)^2+1$  points. For $m\le 3$, such a curve
exists  since $2m-1$ points determine a pencil
of curves of bidegree $(m-1,1)$.
For the case $m=4$, consider $8$ points in $\RP^2$ such that there
exists a non-real rational cubic $C_0\subset\C P^2$
passing through these 8 points (such
configuration of $8$ points exist). Since $C_0$ has a unique nodal
point, it has to be non-real. Furthermore, since $C_0$ intersects $\RP^2$
in an odd number of points, it has to intersect $\RP^2$ in a ninth point.
Hence the union of $C_0$ with its complex
conjugate is a real algebraic curve of degree 6 with 9 solitary
points and two complex conjugate nodal points. Denote by $O$ the line
passing through
the two
latter.
Blowing up
the two nodes
and blowing down the strict transform of $O$, we obtain a real
algebraic curve of bidegree $(4,4)$  in the quadric ellipsoid whose
real part has exactly 10 points.

The case $m=5$ is treated by applying the deformation to the normal
cone construction to  a non-singular real hyperplane section $B$, with
$\R B\ne \varnothing$, in the quadric ellipsoid $X$. Here we use notations
from Section \ref{normal_cone}.
 According to Proposition
\ref{prop:deg4}, there exists a real algebraic curve $C_B$ of
bidegree $(4,2)$ in $E_B=\Sigma_2$ whose real part consists of 14 solitary
nodes. Let $C_X$ be a reducible curve of bidegree $(1,1)$ in $X$
passing through $X\cap E_B\cap C_B$, and let us define $C_0=C_X\cup
C_B$. The curve $C_0$ is a finite real algebraic curve with
$|\R C_0|=15$, hence Proposition \ref{prop:dnc1} ensures the
existence of a finite real algebraic  curve $C$
of bidegree $(5,5)$ in $X$ with $|\R C|=15$.

We now prove that there does not exist finite real algebraic curves
of bidegree $m\le 5$ with a number of real points greater than the one
stated in the proposition.
 By B\'{e}zout Theorem, a finite real algebraic curve of bidegree $(m,m)$ with $m=1$ or
 $m=2$ has at most $1$ or $2$ real points respectively. According to
 Theorem \ref{thm:quad}, a finite real algebraic curve of bidegree
 $(3,3)$, $(4,4)$ or $(5,5)$ in the quadric ellipsoid cannot have more
 that $5, 10$, or $16$ real points respectively. Suppose that there
 exists a real algebraic curve  of bidegree
  $(5,5)$ in the quadric ellipsoid with $16$ real points. By the genus
 formula, this curve is rational and its 16 real points are all
 ordinary nodes. By a small perturbation creating an oval for each
 node, we obtain a non-singular real algebraic curve of bidegree
  $(5,5)$ in the quadric ellipsoid  whose real part consists
 of exactly 16 connected components, each of them bounding a disc in
 the sphere. This contradicts the congruence \cite[Theorem 1b)]{Mik91}.
\end{proof}

\bibliographystyle{alpha}
\bibliography{Biblio}
\end{document}

%% file: Figures/03-j.inc
\setlength{\unitlength}{1cm}%
\noindent{}\begin{picture}(5.675,0.595303)(0,0)%
\put(0,0){\includegraphics{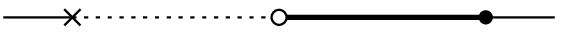}}
\put(0.64667,0.0352778){\makebox(0,0)[lb]{{\normalsize $0$}}}
\put(2.74667,0.0352778){\makebox(0,0)[lb]{{\normalsize $1$}}}
\put(4.75806,0.0732525){\makebox(0,0)[lb]{{\normalsize $\infty$}}}
\end{picture}%

%% file: Figures/curve.inc
\setlength{\unitlength}{1cm}%
\noindent{}\begin{picture}(8.72964,5.48776)(0,0)%
\put(0,0){\includegraphics{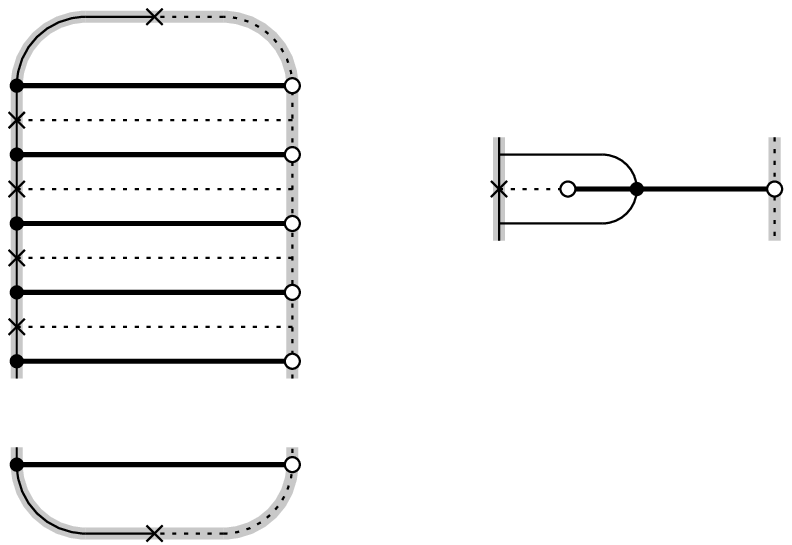}}
\put(0.0352778,4.23649){\makebox(0,0)[lb]{{\normalsize $\scriptstyle1$}}}
\put(3.5241,3.19472){\makebox(0,0)[lb]{{\normalsize $\mathstrut\bigg\}$}}}
\put(0.224096,3.19472){\makebox(0,0)[lb]{{\normalsize $\mathstrut\bigg\{$}}}
\put(0.0352778,3.53649){\makebox(0,0)[lb]{{\normalsize $\scriptstyle2$}}}
\put(0.0352778,2.83649){\makebox(0,0)[lb]{{\normalsize $\scriptstyle3$}}}
\put(3.5241,1.79472){\makebox(0,0)[lb]{{\normalsize $\mathstrut\bigg\}$}}}
\put(0.224096,1.79472){\makebox(0,0)[lb]{{\normalsize $\mathstrut\bigg\{$}}}
\put(0.0352778,2.13649){\makebox(0,0)[lb]{{\normalsize $\scriptstyle4$}}}
\put(0.556375,1.16544){\makebox(0,0)[lb]{{\normalsize $\vdots$}}}
\put(1.95637,1.16544){\makebox(0,0)[lb]{{\normalsize $\vdots$}}}
\put(3.35637,1.16544){\makebox(0,0)[lb]{{\normalsize $\vdots$}}}
\put(5.1241,3.19472){\makebox(0,0)[lb]{{\normalsize $\mathstrut\bigg\{$}}}
\put(8.4241,3.19472){\makebox(0,0)[lb]{{\normalsize $\mathstrut\bigg\}$}}}
\put(4.1722,3.55126){\makebox(0,0)[lb]{{\normalsize $\Leftarrow\joinrel\Relbar$}}}
\end{picture}%

%% file: Figures/curve_g.inc
\setlength{\unitlength}{1cm}%
\noindent{}\begin{picture}(7.38323,5.325)(0,0)%
\put(0,0){\includegraphics{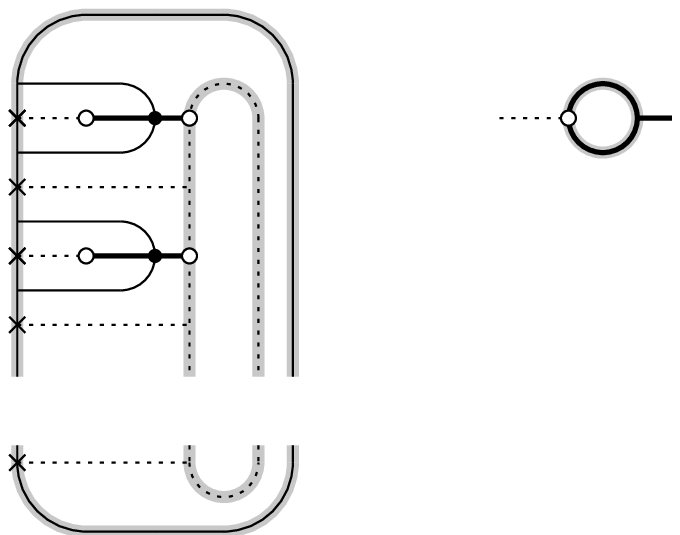}}
\put(0.642873,1.12156){\makebox(0,0)[lb]{\color[rgb]{0.784314,0.784314,0.784314}{\normalsize $\vdots$}}}
\put(3.44287,1.12156){\makebox(0,0)[lb]{\color[rgb]{0.784314,0.784314,0.784314}{\normalsize $\vdots$}}}
\put(0.642873,1.12156){\makebox(0,0)[lb]{{\normalsize $\vdots$}}}
\put(3.44287,1.12156){\makebox(0,0)[lb]{{\normalsize $\vdots$}}}
\put(2.39287,1.12156){\makebox(0,0)[lb]{{\normalsize $\vdots$}}}
\put(3.09287,1.12156){\makebox(0,0)[lb]{{\normalsize $\vdots$}}}
\put(0.121776,4.15511){\makebox(0,0)[lb]{{\normalsize $\scriptstyle1$}}}
\put(0.121776,3.45511){\makebox(0,0)[lb]{{\normalsize $\scriptstyle2$}}}
\put(0.121776,2.75511){\makebox(0,0)[lb]{{\normalsize $\scriptstyle3$}}}
\put(0.121776,2.05511){\makebox(0,0)[lb]{{\normalsize $\scriptstyle4$}}}
\put(0.0352778,0.655109){\makebox(0,0)[lb]{{\normalsize $\scriptstyle2n$}}}
\put(4.2587,4.16988){\makebox(0,0)[lb]{{\normalsize $\Leftarrow\joinrel\Relbar$}}}
\end{picture}%